\documentclass[11pt]{amsart}
\usepackage{amsopn}
\usepackage{amssymb, amscd}
\usepackage{multirow}
\usepackage{graphicx, graphics, epsfig}
\usepackage{faktor}
\usepackage{enumerate}
\usepackage{tikz}
\usepackage{pgfplots}
\usepackage{hyperref}

\usepackage{paracol}
\globalcounter{figure}

\usepackage{color}

\newcommand{\nc}{\newcommand}

\nc{\fg}{\mathfrak{f} } \nc{\vg}{\mathfrak{v} } \nc{\wg}{\mathfrak{w} }
\nc{\zg}{\mathfrak{z} } \nc{\ngo}{\mathfrak{n} } \nc{\kg}{\mathfrak{k} }
\nc{\mg}{\mathfrak{m} } \nc{\bg}{\mathfrak{b} } \nc{\ggo}{\mathfrak{g} }
\nc{\ggob}{\overline{\mathfrak{g}} } \nc{\sog}{\mathfrak{so} }
\nc{\sug}{\mathfrak{su} } \nc{\spg}{\mathfrak{sp} } \nc{\slg}{\mathfrak{sl} }
\nc{\glg}{\mathfrak{gl} } \nc{\cg}{\mathfrak{c} } \nc{\rg}{\mathfrak{r} }
\nc{\hg}{\mathfrak{h} } \nc{\tg}{\mathfrak{t} } \nc{\ug}{\mathfrak{u} }
\nc{\dg}{\mathfrak{d} } \nc{\ag}{\mathfrak{a} } \nc{\pg}{\mathfrak{p} }
\nc{\sg}{\mathfrak{s} } \nc{\affg}{\mathfrak{aff} } \nc{\qg}{\mathfrak{q} } \nc{\lgo}{\mathfrak{l} }

\nc{\pca}{\mathcal{P}} \nc{\nca}{\mathcal{N}} \nc{\lca}{\mathcal{L}}
\nc{\oca}{\mathcal{O}} \nc{\mca}{\mathcal{M}} \nc{\tca}{\mathcal{T}}
\nc{\aca}{\mathcal{A}} \nc{\cca}{\mathcal{C}} \nc{\gca}{\mathcal{G}}
\nc{\sca}{\mathcal{S}} \nc{\hca}{\mathcal{H}} \nc{\bca}{\mathcal{B}}
\nc{\dca}{\mathcal{D}} \nc{\val}{\operatorname{val}}

\nc{\vp}{\varphi} \nc{\ddt}{\tfrac{d}{dt}} \nc{\dsdt}{\tfrac{d^2}{dt^2}} \nc{\dds}{\frac{d}{ds}}
\nc{\dpar}{\frac{\partial}{\partial t}} \nc{\im}{\mathrm{i}}

\nc{\SO}{\mathrm{SO}} \nc{\Spe}{\mathrm{Sp}} \nc{\Sl}{\mathrm{SL}}
\nc{\SU}{\mathrm{SU}} \nc{\Or}{\mathrm{O}} \nc{\U}{\mathrm{U}} \nc{\Gl}{\mathrm{GL}}
\nc{\Se}{\mathrm{S}} \nc{\Cl}{\mathrm{Cl}} \nc{\Spein}{\mathrm{Spin}}
\nc{\Pin}{\mathrm{Pin}} \nc{\G}{\mathrm{GL}_n(\RR)} \nc{\g}{\mathfrak{gl}_n(\RR)}

\nc{\RR}{{\Bbb R}} \nc{\HH}{{\Bbb H}} \nc{\CC}{{\Bbb C}} \nc{\ZZ}{{\Bbb Z}} \nc{\SSS}{{\Bbb S}}
\nc{\FF}{{\Bbb F}} \nc{\NN}{{\Bbb N}} \nc{\QQ}{{\Bbb Q}} \nc{\PP}{{\Bbb P}} \nc{\OO}{{\Bbb O}}

\nc{\vs}{\vspace{.2cm}} \nc{\vsp}{\vspace{1cm}} \nc{\ip}{\langle\cdot,\cdot\rangle}
\nc{\ipp}{(\cdot,\cdot)} \nc{\la}{\langle} \nc{\ra}{\rangle} \nc{\unm}{\tfrac{1}{2}}
\nc{\unc}{\tfrac{1}{4}} \nc{\und}{\tfrac{1}{16}} \nc{\no}{\vs\noindent}
\nc{\lam}{\Lambda^2(\RR^n)^*\otimes\RR^n} \nc{\tangz}{{\rm T}^{\rm Zar}}
\nc{\nor}{{\sf n}}  \nc{\mum}{/\!\!/} \nc{\kir}{/\!\!/\!\!/}
\nc{\Ri}{\tfrac{4\Ric_{\mu}}{||\mu||^2}} \nc{\ds}{\displaystyle}
\nc{\ben}{\begin{enumerate}} \nc{\een}{\end{enumerate}} \nc{\f}{\frac}
\nc{\lb}{[\cdot,\cdot]} \nc{\isn}{\tfrac{1}{||v||^2}}
\nc{\gkp}{(\ggo=\kg\oplus\pg,\ip)} \nc{\ukh}{(\ug=\kg\oplus\hg,\ip)}
\nc{\tgkp}{(\tilde{\ggo}=\kg\oplus\pg,\ip)}
\nc{\wt}{\widetilde}
\nc{\iop}{\mathtt{i}} \nc{\jop}{\mathtt{j}}

\nc{\Hess}{\operatorname{Hess}} \nc{\ad}{\operatorname{ad}}
\nc{\Ad}{\operatorname{Ad}} \nc{\rank}{\operatorname{rk}}
\nc{\Irr}{\operatorname{Irr}} \nc{\End}{\operatorname{End}}
\nc{\Aut}{\operatorname{Aut}} \nc{\Inn}{\operatorname{Inn}}
\nc{\Der}{\operatorname{Der}} \nc{\Ker}{\operatorname{Ker}}
\nc{\Iso}{\operatorname{Iso}} \nc{\Diff}{\operatorname{Diff}}
\nc{\Lie}{\operatorname{L}} \nc{\tr}{\operatorname{tr}} \nc{\dif}{\operatorname{d}}
\nc{\sen}{\operatorname{sen}} \nc{\modu}{\operatorname{mod}}
\nc{\CRic}{\operatorname{PP}} \nc{\Cric}{\operatorname{P}} \nc{\Ricci}{\operatorname{Ric}}
\nc{\sym}{\operatorname{sym}} \nc{\herm}{\operatorname{herm}} \nc{\symac}{\operatorname{sym^{ac}}}
\nc{\symc}{\operatorname{sym^{c}}} \nc{\scalar}{\operatorname{Sc}}
\nc{\grad}{\operatorname{grad}} \nc{\ricci}{\operatorname{Rc}} \nc{\kil}{\operatorname{B}} \nc{\cas}{\operatorname{C}}
\nc{\Nor}{\operatorname{Norm}}  \nc{\ricc}{\operatorname{Rc^{c}}}
\nc{\Ricc}{\operatorname{Ric^{c}}} \nc{\ricac}{\operatorname{Rc^{ac}}}
\nc{\Ricac}{\operatorname{Ric^{ac}}} \nc{\Riem}{\operatorname{Rm}} \nc{\Sec}{\operatorname{Sec}}
\nc{\riccig}{\operatorname{ric^{\gamma}}} \nc{\mm}{\operatorname{m}} \nc{\Mm}{\operatorname{M}}
\nc{\Le}{\operatorname{L}} \nc{\tang}{\operatorname{T}}
\nc{\level}{\operatorname{level}} \nc{\rad}{\operatorname{r}}
\nc{\abel}{\operatorname{ab}} \nc{\CH}{\operatorname{CH}} \nc{\Cone}{{\mathcal C}} \nc{\CCone}{\operatorname{CC}} \nc{\CP}{{\mathcal P}}
\nc{\mcc}{\operatorname{mcc}} \nc{\Adj}{\operatorname{Adj}}
\nc{\Order}{\operatorname{O}}  \nc{\inj}{\operatorname{inj}} \nc{\proy}{\operatorname{pr}}
\nc{\vol}{\operatorname{vol}} \nc{\Diag}{\operatorname{Dg}} \nc{\Diagg}{\operatorname{Diag}}
\nc{\Spec}{\operatorname{Spec}} \nc{\Ima}{\operatorname{Im}} \nc{\Rea}{\operatorname{Re}}
\nc{\spann}{\operatorname{span}} \nc{\Aff}{\operatorname{Aff}} \nc{\E}{\operatorname{E}} \nc{\id}{\operatorname{id}} \nc{\dete}{\operatorname{det}} \nc{\Crit}{\operatorname{Crit}}

\theoremstyle{plain}
\newtheorem{theorem}{Theorem}[section]
\newtheorem{proposition}[theorem]{Proposition}
\newtheorem{corollary}[theorem]{Corollary}
\newtheorem{lemma}[theorem]{Lemma}

\theoremstyle{definition}
\newtheorem{definition}[theorem]{Definition}

\theoremstyle{remark}
\newtheorem{remark}[theorem]{Remark}
\newtheorem{example}[theorem]{Example}

\title{Prescribing Ricci curvature on homogeneous spaces}

\author{Jorge Lauret} \author{Cynthia E. Will}

\address{FaMAF, Universidad Nacional de C\'ordoba and CIEM, CONICET (Argentina)}
\email{jorgelauret@unc.edu.ar} \email{cynthia.will@unc.edu.ar}

\thanks{This research was partially supported by a grant from Univ. Nac. de C\'ordoba (Argentina).}

\begin{document}

\maketitle

\begin{abstract}
The prescribed Ricci curvature problem in the context of $G$-invariant metrics on a homogeneous space $M=G/K$ is studied.  We focus on the metrics at which the map $g\mapsto\ricci(g)$ is, locally, as injective and surjective as it can be.  Our main result is that such property is generic in the compact case.  Our main tool is a formula for the Lichnerowicz Laplacian we prove in terms of the moment map for the variety of algebras.     
\end{abstract}

\tableofcontents

\section{Introduction}\label{intro}

The prescribed Ricci curvature problem on a differentiable manifold $M$ is a classical problem in geometric analysis (see e.g.\ \cite[Chapter 5]{Bss}):  given a symmetric $2$-tensor $T$, one asks about the existence and uniqueness of a Riemannian metric $g$ (up to scaling) and a constant $c>0$ such that
\begin{equation}\label{PRP-intro}
\ricci(g) = cT, 
\end{equation}
where $\ricci(g)$ denotes the Ricci tensor of $g$.  The role of this in principle auxiliary constant $c$ is to compensate the scaling invariance of $\ricci$, however, the possibility of having solutions with different constants for the same $T$ raises a subtle problem.  For a compact $M$, the set of such constants for a fixed $T>0$ is bounded above (see \cite{DtrKso}).  

There are many classical results on the prescribed Ricci curvature problem (see the recent survey \cite{BttPlm}), we found the following one particularly inspiring.  It was first proved by DeTurck and then independently by Hamilton using different methods.  

\begin{theorem}\cite{Dtr1}, \cite[Theorem 5.1]{Hml}.  
Let $\overline{g}$ be a metric on $\SSS^n$ of constant curvature $+1$, so $\ricci(\overline{g})=\overline{g}$.  Then the image of a neighborhood of $\overline{g}$ under the map $\ricci$ is a submanifold of codimension one in a neighborhood of $\overline{g}$.  For every $T$ near $\overline{g}$ there exists a unique constant $c$ such that $\ricci(g)=cT$ for some $g$ near $\overline{g}$, and g is also the unique solution in the
neighborhood of $\overline{g}$ if we normalize the volume so $V(g) = V(\overline{g})$.
\end{theorem}

This was later proved for any irreducible symmetric space of compact type in \cite{Dtr2}, actually for any Einstein metric $g$ on a compact $M$ whose Lichnerowicz Laplacian has a one-dimensional kernel (e.g., if $g$ is de Rham irreducible and has nonnegative sectional curvature).  It is also satisfied by some non-compact manifolds, including the real and complex hyperbolic spaces (see \cite{Dly}).  

In this paper, we study the prescribed Ricci curvature problem in a homogeneous context.  More precisely, we fix a homogeneous manifold $M$ and a Lie group $G$ acting transitively on $M$ (or equivalently, a homogeneous space $M=G/K$, where $K$ is the isotropy subgroup of $G$ at a point $o\in M$) and consider equation \eqref{PRP-intro} for a $G$-invariant tensor $T$ and metric $g$.  The question is therefore about the image and the injectivity (up to scaling) of the function
$$
\ricci:\mca^G\longrightarrow\sca^2(M)^G,  
$$
where $\sca^2(M)^G$ and $\mca^G$ respectively denote the finite-dimensional vector space of all $G$-invariant symmetric $2$-tensors and the open subset of all $G$-invariant metrics on $M$.  By fixing a background metric $\overline{g}\in\mca^G$, the number $\dete_{\overline{g}}{g}$ plays the same role as the volume of $g\in\mca^G$ beyond the compact case.  

We aim to understand the plausibility of the following concept, whose definition is strongly motivated by the above theorem.  

\begin{definition} (See Definition \ref{RLI-hom}). 
A metric $g_0\in\mca^G$ is said to be {\it Ricci locally invertible} if there exist an open neighborhood $U_1$ of $g_0$ in the submanifold 
$$
\mca_0^G:=\{ g\in\mca^G:\dete_{\overline{g}}{g}=\dete_{\overline{g}}{g_0}\}
$$ 
and an open neighborhood $V$ of $\ricci(g_0)$ in $\sca^2(M)^G$, such that the following conditions hold: 

\begin{enumerate}[{\rm (a)}]
\item $\ricci(U_1)$ is a submanifold of codimension one of $V$ and $\ricci:U_1\rightarrow\ricci(U_1)$ is a diffeomorphism; 

\item for any $T\in V$, there exists a unique pair $(g,c)$ with $g\in U_1$ and $c>0$ such that $\ricci(g)=cT$.   
\end{enumerate} 
\end{definition}

In other words, near $g_0$, the function $\ricci$ is as invertible as it can be, and for any $T$ near $\ricci(g_0)$, existence and (local) uniqueness of solutions to the prescribed Ricci curvature problem \eqref{PRP-intro} are guaranteed.   
The subset 
$$
\mca^G_{inv}:=\left\{ g\in\mca^G:g\;\mbox{is Ricci locally invertible}\right\}
$$
is open and $\Aut(G/K)$-invariant; furthermore, the subset $\RR_ +\ricci(\mca^G_{inv})$ is open in $\sca^2(M)^G$ by condition (b).  Regarding the study of equation \eqref{PRP-intro}, the question of whether $\mca^G_{inv}$ is dense in $\mca^G$ arises.  It follows from condition (a) that, 
$$
\mca^G_{inv}\subset\mca^G_{\ricci}:=\left\{ g\in\mca^G:\Ker d\ricci|_g=\RR g\right\},  
$$
the set of maximal rank metrics.  It is also natural to wonder about the relationship between Ricci local invertibility and the invertibility of the map
$$
\widetilde{\ricci}:\mca^G\longrightarrow\sca^2(M)^G, \qquad \widetilde{\ricci}(g):=(\dete_{\overline{g}}{g})\ricci(g),
$$
breaking the scaling invariance of $\ricci$.  The subset of metrics
$$ 
\mca^G_{\widetilde{\ricci}}:=\left\{ g\in\mca^G:\widetilde{\ricci}\; \mbox{is a local diffeomorphism at}\; g\right\}
$$
appears as more tractable than $\mca^G_{inv}$.  Since $\ricci(g)$ is a rational function in the coordinates of $g\in\mca^G$, both subsets $\mca^G_{\widetilde{\ricci}}$ and $\mca^G_{\ricci}$ are either empty or open and dense in $\mca^G$, and the subset 
$$
\mca^G_{\scalar}:=\left\{ g\in\mca^G:\scalar(g)\ne 0\right\}
$$
is always open and dense unless $G$ is abelian (see Lemma \ref{RLI-lem}), where $\scalar(g)$ is the scalar curvature of $g$.  

We assume from now on that $G$ is unimodular.  The main results of the present paper are described by the following five theorems.    

\begin{theorem} (See Theorem \ref{RLI-thm}).  
$\mca^G_{\widetilde{\ricci}}=\mca^G_{\ricci}\cap\mca^G_{\scalar}\subset\mca^G_{inv}\subset \mca^G_{\ricci}$.  Moreover, if $\mca^G_{\ricci}$ is nonempty, then $\mca^G_{inv}$ is so.  
\end{theorem}

This theorem implies that $\mca^G_{inv}$ is either empty or open and dense in $\mca^G$.  What is most likely?  Note that, consequently, the open subset $\RR_+\ricci(\mca^G_{inv})=\widetilde{\ricci}(\mca^G_{inv})$ is also either empty or dense in $\RR_+\ricci(\mca^G)=\widetilde{\ricci}(\mca^G)$, the set of \eqref{PRP-intro}-solvable $T$'s.   
    
If $M=G_1/K_1\times G_2/K_2$, $G=G_1\times G_2$ and the isotropy representation of $G_1/K_1$ does not contain the trivial subrepresentation, then $\mca^G_{inv}$ is empty.  This condition is not even necessary, we found that $\mca^G_{inv}$ is also empty for $M=\SSS^5\times\SSS^1$, $G=\SU(3)\times S^1$.  However, our second main result shows that emptynness of $\mca^G_{inv}$ is less common than one may expect.   

\begin{theorem} (See Theorem \ref{dRic-thm2}). 
Let $g\in\mca^G$ be a naturally reductive metric with respect to $G$.  If $g$ is holonomy irreducible, then $g\in\mca^G_{\ricci}$.  In particular, $g$ is Ricci locally invertible if in addition $\scalar(g)\ne 0$.  
\end{theorem}

This can be applied to the so-called D'Atri-Ziller metrics.  Given a compact Lie group $M=H=G/\Delta K$, where $G=H\times K$ and $K\subset H$ is a nontrivial closed Lie subgroup, $\mca^G$ is identified with the set of all left-invariant metrics on $H$ which are in addition $\Ad(K)$-invariant.  D'Atri and Ziller \cite{DtrZll} proved that any $g\in\mca^G$ is naturally reductive with respect to $G$.  Moreover, they showed that if $\kg$ can not be decomposed as $\kg=\kg\cap\hg_1\oplus\kg\cap\hg_2$ for nonzero ideals $\hg_1$ and $\hg_2$ of $\hg$ such that $\hg=\hg_1\oplus\hg_2$ (e.g., if $H$ is simple), then they are all holonomy irreducible and hence $\mca^G=\mca^G_{\ricci}$ and $\mca^G_{\scalar}\subset\mca^G_{inv}$ by the above theorem.  In particular, $\RR_+\ricci(\mca^G)\cap\mca^G$ is open. 

In Section \ref{natred-sec}, departing from classical results due to Kostant \cite{Kst2} and D'Atri and Ziller \cite{DtrZll} on the geometry and algebraic aspects of naturally reductive spaces, we prove some extra technical properties that allows us to conclude that $\mca^G_{inv}$ is nonempty for most compact homogeneous spaces.    

\begin{theorem} (See Theorem \ref{cor-dense}). 
Assume that a homogeneous space $M=G/K$ admits a naturally reductive metric with respect to $G$ (e.g., if $G$ is compact).  If there exist no nonzero ideals $\ggo_1$ and $\ggo_2$ of $\ggo$ such that $\ggo=\ggo_1\oplus\ggo_2$ and $\kg=\kg\cap\ggo_1\oplus\kg\cap\ggo_2$ (e.g., if $\ggo$ is indecomposable), then $\mca^G_{inv}$ is open and dense in $\mca^G$.  
\end{theorem}

The indecomposability condition on $\kg$ in the above theorem is not necessary, each of the following decomposable examples have $\mca^G_{inv}$ open and dense: $M=\SSS^5\times\SSS^5$ and $G=\SU(3)\times \SU(3)$, $M=\SSS^7\times\SSS^5$ and $G=\SU(4)\times \SU(3)$ and the Lie group case $M=\SSS^3\times\SSS^1$ and $G=\SU(2)\times S^1$ (see Section \ref{exa-sec2}).  

The following formula for the derivative of the function $\ricci$ is proved in Section \ref{mba2-sec} via the moving bracket approach and represents a very useful tool in the paper.  Let $\ggo=\kg\oplus\pg$ be any reductive decomposition for $M=G/K$ and let $\lb_\pg\in\Lambda^2\pg^*\otimes\pg$ denote the algebra product defined by the Lie bracket $\lb$ of $\ggo$ on $\pg=T_oM$.  Given $g\in\mca^G$, $\ip:=g_o$, the moment map $\Mm\in\glg(\pg)$ at $\lb_\pg$ for the usual $\glg(\pg)$-representation   
$$
A\cdot\lambda:=A\lambda(\cdot,\cdot) - \lambda(A\cdot,\cdot) - \lambda(\cdot,A\cdot), \qquad \forall A\in\glg(\pg), \quad \lambda\in\Lambda^2\pg^*\otimes\pg,
$$
is defined by 
$$
\tr{\Mm A} = \unc \la A\cdot\lb_{\pg}, \lb_{\pg}\ra, \qquad \forall A\in\glg(\pg), 
$$ 
and satisfies that  $\ricci(g)=\la\Mm\cdot,\cdot\ra-\unm\kil|_{\pg\times\pg}$, where $\kil$ is the Killing form of $\ggo$.  Recall also that $\mca^G$ is naturally endowed with a Riemannian metric, defined by $\la T,T\ra_g:=\sum T(X_i,X_i)^2$ for any $g_o$-orthonormal basis $\{ X_i\}$ of $T_oM$.

\begin{theorem} (See Lemma \ref{dRicmu}).  
At each $g\in\mca^G$, the derivative 
$$
d\ricci|_g:\sca^2(M)^G\rightarrow\sca^2(M)^G
$$ 
is a self adjoint operator whose image is $\ip_g$-orthogonal to $\RR g$ and is given by 
$$
\la d\ricci|_gT,T\ra = \unc \left|A\cdot\lb_{\pg}\right|^2 + \tr{\Mm A^2}, \qquad\forall T=\la A\cdot,\cdot\ra\in\sca^2(M)^G \; (\mbox{i.e.}, \, A^t=A).      
$$
\end{theorem}

We note that if $G$ is compact, then this is providing a tool to compute the Lichnerowicz Laplacian $\Delta_L$ of the metric $g$; indeed, $\Delta_L T= 2\, d\ricci|_gT$ for any divergence-free $G$-invariant symmetric $2$-tensor $T$.   

The above formula drastically simplifies in the naturally reductive case (see Lemma \ref{dRic}) and for the left-invariant metric determined by the Killing form on any (not necessarily compact) semisimple Lie group (see Lemma \ref{dRc-simple}):
$$
\la d\ricci|_gT,T\ra = \unc \sum \left| [\ad_\pg{X_i},A]\right|^2, \qquad\forall T=\la A\cdot,\cdot\ra\in\sca^2(M)^G,      
$$ 
where $\ad_\pg{X}:=[X,\cdot]_\pg\in\glg(\pg)$.  This was a key tool in the study of Ricci locally invertibility in the naturally reductive case and also provides the following application.  

Let $M=G$ be a non-compact simple Lie group and for any Cartan decomposition $\ggo=\hg\oplus\qg$, consider the metric $g_{\kil}\in\mca^G$ defined by $\ip:=-\kil|_{\hg\times\hg} + \kil|_{\qg\times\qg}$, where $\kil$ denotes the Killing form of $\ggo$.  Note that $\mca^G$ is in this case quite large, it is the $\frac{n(n+1)}{2}$-parametric set of all left-invariant metrics on $G$, where $n=\dim{M}$.   

\begin{theorem} (See Corollary \ref{simple-thm}).  
$g_{\kil}\in\mca^G_{\ricci}$, the set $\mca^G_{inv}$ is open and dense in $\mca^G$ and if in addition $\scalar(g_{\kil})\ne 0$, then $g_{\kil}$ is Ricci locally invertible.   
\end{theorem}

Many natural open questions on the homogeneous prescribed Ricci curvature problem and plenty of examples illustrating the different features of the concept of Ricci local invertibility have been included throughout the paper.  

We refer to the survey \cite{BttPlm} and the references therein for the study of the prescribed Ricci curvature problem via a variational principle on many different classes of homogeneous spaces, including spheres and projective spaces (see \cite{BttPlmRbnZll}), generalized Wallach spaces (see \cite{Plm2}), generalized flag manifolds (see \cite{GldPlm}), spaces with two irreducible isotropy summands (see \cite{Plm,PlmRbn}), $3$-dimensional unimodular Lie groups (see \cite{Btt}) and D'Atri-Ziller metrics on compact (see \cite{ArrPlmZll}) and non-compact (see \cite{ArrGldPlm}) simple Lie groups.  Unlike here, these papers focus on a complete description of the set $\RR_+\ricci(\mca^G)$ and its intersection with $\mca^G$.

\vs \noindent {\it Acknowledgements.}  We are grateful with Marcos Salvai and the anonymous referee for very helpful comments.

\section{Preliminaries}\label{preli}

Let $M$ be a connected differentiable manifold of dimension $n$.  We assume that $M$ is homogeneous and fix $G$, one of the connected Lie groups acting transitively on $M$.  If $K\subset G$ is the isotropy subgroup at a point $o\in M$, it will  tacitly be assumed that the homogeneous space $G/K$ is almost-effective (i.e., the subgroup of $G$ of those elements acting trivially is discrete).  The existence of $G$-invariant Riemannian metrics on $M=G/K$ is therefore equivalent to the compactness of $\overline{\Ad(K)}\subset\glg(\ggo)$, where $\ggo$ is the Lie algebra of $G$.

In order to set a background for the study of the $G$-invariant geometry of $M$, including tensorial and curvature computations, it is useful to consider a {\it reductive decomposition} $\ggo=\kg\oplus\pg$ (i.e., $\Ad(K)\pg\subset\pg$) for the homogeneous space $G/K$, where $\kg$ is the Lie algebra of $K$.  In this way, since $\Ker d\pi|_e=\kg$, where $\pi:G\rightarrow G/K$ is the usual projection, the tangent space at the origin point $o\in M$ is identified with $\pg$, $T_oM\equiv\pg$, via the isomorphism $\pg\longrightarrow T_oM$, $X\mapsto d\pi|_eX=X_o$.  Each $X\in \ggo$ can also be viewed as the vector field on $M$ defined by $X_p:=\left. \ddt\right |_{t=0} \exp{tX}\cdot p$.

This also provides identifications for the finite-dimensional vector spaces of $G$-invariant tensor fields of different types by evaluating at the point $o$.  For instance, the space of all $G$-invariant symmetric $2$-forms on $M$ can be parametrized by 
$$
\sca^2(M)^G\leftrightarrow\sym^2(\pg)^K, 
$$
where $\sym^2(\pg)^K$ denotes the space of all $\Ad(K)$-invariant symmetric $2$-forms on the $n$-dimensional vector space $\pg$.  Thus for the set $\mca^G$ of all $G$-invariant metrics on $M=G/K$ one obtains the identification,
$$
\mca^G\leftrightarrow\sym_+^2(\pg)^K, \qquad g \leftrightarrow \ip:=g_o,
$$
where $\sym_+^2(\pg)\subset\sym^2(\pg)$ is the open cone of positive definite symmetric $2$-forms.

By fixing a background metric $g\in\mca^G$, $\ip=g_o$, one may use in turn the corresponding identifications in terms of operators,
$$
\sca^2(M)^G\leftrightarrow\sym(\pg)^K, \qquad  \mca^G\leftrightarrow\sym_+(\pg)^K,
$$
where $\sym(\pg)$ is the vector space of all self-adjoint (or symmetric) linear maps of $\pg$ with respect to $\ip$ and $\sym_+(\pg)$ is the open subset of those which are positive definite: 
$$
\sym(\pg)^K \ni A\leftrightarrow T=\la A\cdot,\cdot\ra\in\sym_+^2(\pg)^K, \qquad
\sym_+(\pg)^K \ni h\leftrightarrow \la h\cdot,h\cdot\ra\in\sym_+^2(\pg)^K.
$$
Note that the identity operator $I\in\sym_+(\pg)^K$ always represents the background metric $\ip=g_0$ on $M=G/K$.

The choice of a reductive decomposition may be crucial.  Given a presentation of the homogeneous manifold $M$ as a homogeneous space $M=G/K$, there are in general several reductive decompositions available satisfying different properties which may be more or less useful or appropriate, depending on the questions to be studied.

\begin{remark}
In the case when $K$ is  a discrete subgroup of the center of $G$, we identify $\ggo$ with the space of all left-invariant vector fields on the Lie group $M=G/K$ (rather than with the right-invariant ones as done above) and $\mca^G\leftrightarrow\sym_+^2(\ggo)\leftrightarrow\sym_+(\ggo)$ is the set of all left-invariant metrics on the Lie group $M$.
\end{remark}

\subsection{Aut-isometry}\label{equiv-sec}
Any homogeneous space $M=G/K$ admits a distinguished class of diffeomorphisms, given by the group $\Aut(G/K)$ of all automorphism of $G$ taking $K$ onto $K$.  It is easy to check that $\Aut(G/K)$ acts on $\mca^G$ and two $G$-invariant metrics are said to be {\it aut-isometric} when they belong to the same $\Aut(G/K)$-orbit.  On the other hand, the normalizer $N_G(K)\subset\Aut(G/K)$ acts on $M$ by $n\cdot (a\cdot o)=R_n(a\cdot o):=an\cdot o$.  Note that $R_n$ is an {\it equivariant diffeomorphism} (i.e., $\psi(a\cdot p)=a\cdot \psi(p)$ for all $a\in G$, $p\in M$) and $R_n^*g=I_{n^{-1}}^*g$ for any $n\in N_G(K)$, so $N_G(K)\cdot g$ is contained in the so-called {\it equivariant isometry class} of the metric $g$.  

In the case when $G$ is compact (so $M$ and $K$ are also compact), $\{ R_n:n\in N_G(K)\}$ is the group of all equivariant diffeomorphisms of $M=G/K$ (see \cite[Chapter I, Corollary 4.3]{Brd}).  Thus the $N_G(K)$-orbits are precisely the equivariant isometry classes, and since $N_G(K)$ and $\Aut(G/K)$ have the same connected components, $T_g\Aut(G/K)\cdot g = T_gN_G(K)\cdot g$ for any $g\in\mca^G$.  In many cases, $\Aut(G/K)\cdot g$ is finite, e.g. when $G$ is compact semisimple and $g$ is either normal or the isotropy representation is multiplicity-free (e.g.\ $\rank G=\rank K$).  

At the Lie algebra level, the derivative of each automorphism in $\Aut(G/K)$ belongs to 
$$
\Aut(\ggo/\kg):=\left\{\underline{f}\in\Aut(\ggo):\underline{f}(\kg)=\kg\right\},
$$
and acts on $\sym^2(\pg)^K$ on the left by
\begin{equation}\label{aut-act}
\underline{f}\cdot T:=(f^{-1})^*T = T(f^{-1}\cdot,f^{-1}\cdot), \qquad \mbox{where}\quad \underline{f}=\left[\begin{matrix} \ast&\ast\\ 0&f\end{matrix}\right]\in\Aut(\ggo/\kg).
\end{equation}
If $K$ is connected then the whole group $\Aut(\ggo/\kg)$ acts on $\sym^2(\pg)^K$ by \eqref{aut-act}.  Note that the subgroup $\Ad{(K)}\subset\Aut(\ggo/\kg)$ acts trivially.  The Lie algebra of $\Aut(\ggo/\kg)$ is given by 
$\Der(\ggo/\kg):=\{ \underline{D}\in\Der(\ggo):\underline{D}(\kg)\subset\kg\}$.  In the case when $G$ is simply connected and $K$ connected (in particular, $M$ is simply connected), the $\Aut(\ggo/\kg)$-orbits are precisely the aut-isometry classes since $\Aut(\ggo/\kg)\leftrightarrow\Aut(G/K)$.

\begin{remark}
Two non-aut-isometric $G$-invariant metrics on $M=G/K$ may still be isometric via some $\psi\in\Diff(M)$.
\end{remark}

The moduli space
$$
\sym_+^2(\pg)^K/\Aut(G/K),
$$
parametrizing the set $\mca^G/\Aut(G/K)$ of all $G$-invariant metrics on $M=G/K$ up to aut-isometry is in general hard to compute or understand.

According to \eqref{aut-act}, the $\Aut(G/K)$-action on $\sym_+^2(\pg)^K$ determines the left action on $\sym_+(\pg)^K$ given by, 
$$
\underline{f}\cdot h := \left(\left(f^{-1}\right)^th^2f^{-1}\right)^{1/2}.  
$$
Note that if $f$ is in addition $\ip$-orthogonal, then $\underline{f}\cdot h=fhf^{-1}$ for any $h\in\sym_+(\pg)^K$.

\subsection{Moving-bracket approach}\label{mba-sec}
The curvature of a $G$-invariant metric $g$ on a homogeneous space $M=G/K$ with reductive decomposition $\ggo=\kg\oplus\pg$ is essentially encoded in the inner product $g_o=\ip\in\sym_+^2(\pg)^K$ and the Lie bracket $\mu$ of $\ggo$.  In order to study curvature questions involving the space $\mca^G$, it is therefore natural to vary $\mu$ rather than $\ip$ (see the recent surveys \cite{sol-HS, RNder} and the references therein for further information on this viewpoint).

Recall that we have fixed a background metric $\ip\in\sym_+^2(\pg)^K$, where $\ip=g_o$, $g\in\mca^G$.  Given $h\in\sym_+(\pg)^K$, we denote by $\underline{h}$ the linear map of $\ggo$ defined by $\underline{h}|_\kg:=I$, $\underline{h}|_\pg:=h$ and consider the new Lie algebra $(\ggo,\underline{h}\cdot\mu)$, where $\underline{h}\cdot\mu:=\underline{h}\mu(\underline{h}^{-1}\cdot,\underline{h}^{-1}\cdot)$  is the usual action of $\Gl(\ggo)$ on $\Lambda^2\ggo^*\otimes\ggo$.  In particular, $\underline{h}:(\ggo,\mu)\rightarrow(\ggo,\underline{h}\cdot \mu)$ is a Lie algebra isomorphism.

Let $G_{h\cdot\mu}$ denote a Lie group with Lie algebra $(\ggo,\underline{h}\cdot \mu)$ such that there is an isomorphism $G\rightarrow G_{h\cdot\mu}$ with derivative $\underline{h}$.  Such an isomorphism therefore defines an equivariant isometry
\begin{equation}\label{isom}
(G/K,\la h\cdot,h\cdot\ra) \longrightarrow (G_{h\cdot\mu}/K_{h\cdot\mu},\ip), 
\end{equation}
where $K_{h\cdot\mu}$ is the image of $K$ under the isomorphism (in particular, $K_{h\cdot\mu}$ is a Lie subgroup of $G_{h\cdot\mu}$ with Lie algebra $\kg$).  Note that $\ggo=\kg\oplus\pg$ is a reductive decomposition for all the homogeneous spaces involved.

Our fixed inner product $\ip$ on $\pg$ naturally defines inner products on $\glg(\pg)$ and $\Lambda^2\pg^*\otimes\pg$ by,
$$
\la A,B\ra:= \tr{AB^t}, \qquad \la\lambda,\lambda\ra:=\sum |\ad_\lambda{X_i}|^2 = \sum |\lambda(X_i,X_j)|^2,
$$
where $\{ X_i\}$ will denote from now on an $\ip$-orthonormal basis of $\pg$.  We also introduce the following notation: 
\begin{equation}\label{mup}
\mu_\pg:= \proy_\pg\circ \mu|_{\pg\times\pg} :\pg\times\pg\longrightarrow\pg,
\end{equation}
where $\proy_\pg:\ggo\rightarrow\pg$ is the projection on $\pg$ relative to $\ggo=\kg\oplus\pg$, and  $\ad_\pg{X}:=\ad_{\mu_\pg}{X}$, i.e., the map $Y\mapsto\mu_\pg(X,Y)$ for all $Y\in\pg$.

For any $h\in\sym_+(\pg)^K$, the Ricci operator of $(G_{h\cdot\mu}/K_{h\cdot\mu},\ip)$ is given by
\begin{equation}\label{Ric}
\Ricci_{h\cdot\mu} = \Mm_{h\cdot\mu_\pg} - \unm h^{-1}\kil_\mu h^{-1} - S\left(h\ad_\pg{(h^{-2}H_{\mu_\pg})}h^{-1}\right),
\end{equation}
where $\la\kil_\mu\cdot,\cdot\ra:=\kil_\mu|_{\pg\times\pg}$, $\kil_\mu$ is the Killing form of the Lie algebra $(\ggo,\mu)$ (so $h^{-1}\kil_\mu h^{-1}$ is the Killing form of the Lie algebra $(\ggo,\underline{h}\cdot\mu)$), $\la H_{\mu_\pg},X\ra=\tr{\ad_\pg{X}}$ for all $X\in\pg$, $S:\glg(\pg)\longrightarrow\sym(\pg)$ is the symmetric part operator defined by $S(A):=\unm(A+A^t)$ and
\begin{equation}\label{mm1}
\la\Mm_{h\cdot\mu_\pg},A\ra := \unc\la\theta(A)(h\cdot\mu_\pg),h\cdot\mu_\pg\ra, \qquad\forall A\in\glg(\pg).
\end{equation}
Thus $\Mm$ is precisely the {\it moment map} from geometric invariant theory (see e.g.\ \cite{BhmLfn} and the references therein) for the representation 
$\theta:\glg(\pg)\rightarrow\End(\Lambda^2\pg^*\otimes\pg)$ (derivative of the above usual $\Gl(\pg)$-action) given by,
$$
\theta(A)\lambda := A\lambda(\cdot,\cdot) - \lambda(A\cdot,\cdot) - \lambda(\cdot,A\cdot), \qquad \forall A\in\glg(\pg), \quad \lambda\in\Lambda^2\pg^*\otimes\pg.
$$
Note that $H_{\mu_\pg}=0$ if and only if $G$ is unimodular.  See \cite{alek} for a more detailed treatment of formula \eqref{Ric}.  

In this way, it follows from \eqref{isom} that the Ricci tensor and the Ricci operator of each metric $g_h:=\la h\cdot,h\cdot\ra\in\mca^G$ are respectively given by 
\begin{equation}\label{ric2}
\ricci_{g_h}=\la h\Ricci_{h\cdot\mu}h\cdot,\cdot\ra, \qquad \Ricci_{g_h} = h^{-1}\Ricci_{h\cdot\mu}h, \qquad\forall h\in\sym_+(\pg)^K.
\end{equation}
By \eqref{mm1}, $\tr{\Mm_{h\cdot\mu_\pg}}=-\unc | h\cdot\mu_\pg|^2$, thus the scalar curvature of $(G_{h\cdot\mu}/K_{h\cdot\mu},\ip)$ is given by
\begin{equation}\label{scal}
\scalar_{h\cdot\mu} = -\unc| h\cdot\mu_\pg|^2  - \unm\tr{\kil_\mu}h^{-2} - |h^{-1}H_{\mu_\pg}|^2.
\end{equation}
It is straightforward to check that the following is an alternative definition of the moment map part of the Ricci operator,
\begin{equation}\label{mm2}
\Mm_{h\cdot\mu_\pg} = -\unm\sum (\ad_\pg{h^{-1}X_i})^t\ad_\pg{h^{-1}X_i} + \unc\sum \ad_\pg{h^{-1}X_i}(\ad_\pg{h^{-1}X_i})^t,
\end{equation}
for any $h\in\sym_+(\pg)$.  

We consider the maps
$$
\{ 0\} \xrightarrow[ ]{ } \pg \xrightarrow[\ad_\pg]{ } \glg(\pg) \xrightarrow[\delta_{\mu_{\pg}}]{ } \Lambda^2\pg^*\otimes\pg,
$$
where $\delta_{\mu_{\pg}}(A):=-\theta(A)\mu_\pg$.  Note that $\delta_{\mu_{\pg}}(I)=\mu_{\pg}$, and hence
$$
\delta_{\mu_{\pg}}^t\delta_{\mu_{\pg}}(I)= \delta_{\mu_{\pg}}^t(\mu_{\pg})=-4\Mm_{\mu_{\pg}},
$$
where $\delta_{\mu_{\pg}}^t:\Lambda^2\pg^*\otimes\pg\longrightarrow\glg(\pg)$ is the transpose of $\delta_{\mu_{\pg}}$.

\begin{remark}
When $K$ is trivial, and so $\ggo=\pg$ and $\mu=\mu_\pg$ is the Lie bracket of $\ggo$, this is the first part of the Chevalley cohomology sequence (i.e., Lie algebra cohomology with values in the adjoint representation) of the Lie algebra $(\ggo,\mu)$. In that case, the first cohomology group is given by
$$
H^1(\ggo,\ggo) = \Der(\ggo)/\ad{\ggo} \simeq \Ker \Delta_\mu, \qquad \mbox{where}\quad \Delta_\mu:=\ad_{\mu}\ad_{\mu}^t+\delta_{\mu}^t\delta_{\mu}.
$$
\end{remark}

The first variation of the moment map was computed in \cite{homRF} to study the behavior of homogeneous Ricci flow solutions.  A proof is included for completeness.  

\begin{lemma}\label{dM}\cite[(36)]{homRF}
If $\Mm:\sym_+(\pg)^K\longrightarrow\sym(\pg)^K$ is defined by $\Mm(h):=\Mm_{h\cdot\mu_\pg}$, then
$$
d\Mm|_I = \unm S\circ\delta_{\mu_{\pg}}^t\delta_{\mu_{\pg}}|_{\sym(\pg)^K}.
$$
\end{lemma}

\begin{remark}\label{dM-rem}
In particular, $d\Mm|_I:\sym(\pg)^K\longrightarrow\sym(\pg)^K$ is a self-adjoint operator.
\end{remark}

\begin{proof}
For any $A,B\in\sym(\pg)^K$, if $h(t):=I+tA$, then by \eqref{mm1}, 
\begin{align*}
\la d\Mm|_IA,B\ra =& \left.\ddt\right|_0 \la\Mm_{h(t)\cdot\mu_\pg},B\ra =\unc \left.\ddt\right|_0\la\theta(B)(h(t)\cdot\mu_\pg),h(t)\cdot\mu_\pg\ra \\
=& \unc\la\theta(B)\theta(A)\mu_\pg,\mu_\pg\ra + \unc \la\theta(B)\mu_\pg,\theta(A)\mu_\pg\ra = \unm\la\theta(A)\mu_\pg,\theta(B)\mu_\pg\ra \\
=& \unm\la\delta_{\mu_{\pg}}(A),\delta_{\mu_{\pg}}(B)\ra = \la \unm S\circ\delta_{\mu_{\pg}}^t\delta_{\mu_{\pg}}(A),B\ra,
\end{align*}
concluding the proof.
\end{proof}

\section{Prescribed Ricci curvature problem}\label{preric-sec}

Given a connected differentiable manifold $M^n$ and a connected Lie group $G$ acting on $M$, the $G$-invariant prescribed Ricci curvature problem (PRP for short) is given by the equation
\begin{equation}\label{PRP1}
\ricci_g = T, 
\end{equation}
where $T\in\sca^2(M)^G$ is given and the problem is the existence and uniqueness (up to scaling) of a solution $g\in\mca^G$.  Here $\ricci_g$ denotes the Ricci tensor of the metric $g$.  Equivalently, the PRP deals with the image and injectivity (up to scaling) of the function
$$
\ricci:\mca^G\longrightarrow\sca^2(M)^G, \qquad \ricci(g):=\ricci_g.  
$$
Note that due to the scaling invariance of $\ricci$, i.e., $\ricci(ag)=\ricci(g)$ for any $a>0$, its derivative at a metric $g\in\mca^G$, 
$$
d\ricci|_gS = \left.\ddt\right|_0 \ricci(g+tS), \qquad S\in\sca^2(M)^G,
$$
always has $\RR g\subset\Ker d\ricci|_g$.  It is worth pointing out that the Ricci tensor may have more symmetries than the metric, so the non-existence of a solution to \eqref{PRP1} does not rule out the possibility of having a solution $g\in\mca^H$ for some subgroup $H\subset G$, even when both groups are transitive and compact (see Example \ref{so5-1} below).  

A second version of the PRP reads: given $T\in\sca^2(M)^G$, are there a metric $g\in\mca^G$ and a constant $c>0$ such that
\begin{equation}\label{PRP}
\ricci(g) = cT.
\end{equation}
The role of the constant $c$ is to somehow compensate the scaling invariance of $\ricci$.  The uniqueness, however, of this in principle just auxiliary constant $c$ raises a very subtle problem.  On the contrary, the existence of a solution to the PRP \eqref{PRP} is still equivalent to understand the image of the function $\ricci$, now up to scaling, in the sense that there is solution for a given $T\in\sca^2(M)^G$ if and only if $T\in\RR_+\ricci(\mca^G)$, where $\RR_+:=\{ a\in\RR:a>0\}$.  

We assume from now on that $G$ acts transitively on $M$, which reduces everything to a finite-dimensional setting.  If nonempty, $\mca^G$ is a connected differentiable manifold (actually a symmetric subspace of a symmetric space) which is open in the finite-dimensional vector space $\sca^2(M)^G$ and has
$$
1\leq \dim{\mca^G}\leq n(n+1)/2.
$$
Note that equality holds on the left if and only if the homogeneous space $M=G/K$ is isotropy irreducible, where $K$ denotes the isotropy subgroup of $G$ at some point $o\in M$, and it does on the right if and only if $K$ is a discrete subgroup of the center of $G$ and so $\mca^G$ is the space of all left-invariant metrics on the Lie group $M=G/K$.  

In order to simplify the statements of many properties, questions and results, it is convenient to break the scaling invariance of $\ricci$ by introducing the function
$$
\widetilde{\ricci}:\mca^G\longrightarrow\sca^2(M)^G, \qquad \widetilde{\ricci}(g):=(\dete_{\overline{g}}{g})\ricci(g),
$$
where $\dete_{\overline{g}}{g}$ denotes the determinant of $g$ with respect to a fixed background metric $\overline{g}\in\mca^G$.  Recall that $\dete_{\overline{g}}{g}:=\det{A}$ if $g_o=\overline{g}_o(A\cdot,\cdot)$.  In particular, $\widetilde{\ricci}(ag)=a^n\widetilde{\ricci}(g)$ for any $a>0$.  

\begin{remark}
With the same purpose, one may also consider the function 
$$
\mca^G\times\RR_+\longrightarrow\sca^2(M)^G\times\RR, \qquad (g,c)\mapsto\left(c\ricci(g),\dete_{\overline{g}}{g}\right),
$$
as done by Hamilton in \cite[p.\ 60]{Hml} in the compact case with the number $\dete_{\overline{g}}{g}$ replaced by the volume of $M$ relative to $g$.   Another possibility is $g\mapsto\ricci(g)+\Lambda g$ for some fixed constant $\Lambda$, with the following warning: if $-\Lambda\in\Ker d\ricci|_g$ then $\RR g$ is still contained in the kernel of the derivative of the new function (see \cite{Dly}).   
\end{remark}

The existence part of the PRP \eqref{PRP} therefore consists in understanding the subset of $\sca^2(M)^G$ given by
$$
\widetilde{\ricci}(\mca^G)=\RR_+\ricci(\mca^G)=\left\{ T\in\sca^2(M)^G:\mbox{there exists solution to \eqref{PRP}}\right\}.
$$
We note that $\widetilde{\ricci}$ is never surjective as no homogeneous space can admit Ricci positive and Ricci negative invariant metrics at the same time for topological reasons.  It should be pointed out that even though the image $\ricci(\mca^G)$ is known, it may be hard to understand the set $\RR_+\ricci(\mca^G)$ as for instance in deciding whether it contains $\mca^G$ or not, or whether the set $\RR_+\ricci(\mca^G)\cap\mca^G$ is open in $\mca^G$ or not.  

It is easy to prove that the injectivity of $\widetilde{\ricci}$ is equivalent to the uniqueness of solutions in the following sense: for any given $T$, two pairs $(g_1,c_1)$ and $(g_2,c_2)$ are solutions to \eqref{PRP} if and only if $g_2\in\RR_+ g_1$ and $c_2=c_1$.  However, it may be the case that $\widetilde{\ricci}$ is injective in an open subset $U\subset\mca^G$ and not in $\RR_+U$ (see Example \ref{Berger} below).      

The following is a simple, though quite illuminating, explicit example.  

\begin{example}\label{Berger} {\it Berger spheres}.   
Consider the $3$-sphere $M=\SSS^3$ and $G=\SU(2)\times S^1$, $S^1\subset\SU(2)$, so $\mca^G$ consists of those left-invariant metrics on $\SSS^3=\SU(2)$ that are also $S^1$-invariant.  In terms of the ordered basis $\{ X_1,X_2,X_3\}$ of $\sug(2)$ such that $\RR X_1$ is the Lie algebra of $S^1$ and
$$ 
[X_1,X_2]=\tfrac{1}{\sqrt{2}} X_3,\quad [X_1,X_3]=-\tfrac{1}{\sqrt{2}}X_2, \quad [X_2,X_3]=\tfrac{1}{\sqrt{2}} X_1, $$
which is orthonormal with respect to the Killing metric $-\kil_{\sug(2)}$, we have that 
$$
\sca^2(M)^G=\{ (x,y,y):x,y\in\RR\}, \qquad \mca^G=\{ (a,b,b):a,b>0\}.  
$$
Thus the line $g_b:=(1,b,b)\in\mca^G$, $b>0$ covers all $\mca^G$ up to scaling.  Since the Ricci eigenvalues of each metric $g_b$ are given by $\left\{\frac{1}{4b^2},\frac{2b-1}{4b^2},\frac{2b-1}{4b^2}\right\}$ (which follows from a straightforward computation using \eqref{Ric} and \eqref{mm2}, see \cite{Mln}), we obtain that the family of metrics $g_b$, $b>0$ is pairwise non-homothetic as their ratio equals $2b-1$.  Note that $g_1$ is the round metric on $\SSS^3$ and $\ricci(g_b)>0$ if and only if $b>1/2$.  On the other hand, the scalar curvature is given by $\scalar(g_b)=\frac{4b-1}{4b^2}$, so $\scalar(g_b)=0$ if and only if $b=1/4$ (see Figure \ref{berger-fig2}). 

The Ricci tensor is given by
\begin{equation}\label{ric-berger}
\ricci(g_b)=\left(\tfrac{1}{4b^2},\tfrac{2b-1}{4b},\tfrac{2b-1}{4b}\right), \qquad\forall g_b:=(1,b,b)\in\mca^G, \quad b>0, 
\end{equation}
so $\Ker d\ricci|_{g_b}=\RR g_b$ for any $b>0$; indeed, 
$$
d\ricci|_{g_b}(0,1,1) = \left.\ddt\right|_0 \ricci(g_{b+t}) = \left(-\tfrac{1}{2b^3},\tfrac{1}{4b^2},\tfrac{1}{4b^2}\right)\ne 0.
$$
Thus $g_b\in\mca^G_{inv}$ for any $b\ne 1/4$ by Theorem \ref{RLI-thm}.  However, we also obtain that 
\begin{equation}\label{berger1}
\ricci\left(g_{1/2-b}\right)=c_b\ricci(g_b), \quad\mbox{for any}\quad 0<b<\unm, \quad \mbox{where}\quad c_b:= b^2/\left(\unm-b\right)^2. 
\end{equation}

\vspace{.5cm}

\begin{paracol}{2}
\begin{figure}
\begin{tikzpicture}[scale=0.9]
\draw[->,thick] (-1, 0) -- (5, 0) node[right] {$a$};
\draw[->,thick] (0, -1) -- (0, 5) node[above] {$b$};
\draw[ultra thin,color=gray] (-1.5,-1.5) grid (5.5,5.5); 

\draw[-,very thick, magenta] (2,0) -- (2,5.5);
\draw[very thick, fill, blue] (2,1) circle (1pt) node[right] {{\tiny $g_{1/4}$}};
\draw[very thick, fill, violet] (2,2) circle (1pt) node[right] {{\tiny $g_{1/2}$}};
\draw[very thick, fill, orange] (2,4) circle (1pt) node[right] {{\tiny $g_{1}=round$}};
\draw[very thick, fill, red] (2,0.5) circle (1pt) node[right] {{\tiny $g_{1/2-b}$}};
\draw[very thick, fill, red] (2,1.5) circle (1pt) node[right] {{\tiny $g_{b}$}};

\draw[-,dashed] (0,0) -- (5.5,5.5);
\draw (3,3) node[above left] {{\tiny $\ricci>0$}};
\draw[-,dashed] (0,0) -- (5.5,2.75);
\draw (4,2) node[above left] {{\tiny $\scalar>0$}};
\draw (4,2) node[below right] {{\tiny $\scalar<0$}};

\draw[very thick, fill] (2,0) circle (0.5pt) node[below] {{\tiny $1$}};
\draw (0,4) node[left] {{\tiny $1$}};
\draw (0,2) node[left] {{\tiny $\unm$}};
\draw (0,1) node[left] {{\tiny $\unc$}};
\draw (0,0) node[below left] {{\tiny $0$}};

\draw (-1,4.5) node {$\sca^2(M)^G$};
\draw (4.5,0.5) node[magenta] {$\mca^G$};
\draw[fill=magenta, opacity=0.05] (0,5.5) -- (0,0) -- (5.5,0) -- (5.5,5.5) -- cycle;
\end{tikzpicture}
\caption{Berger spheres, $M=\SSS^3$, $G=\SU(2)\times S^1$, $g_b:=(1,b,b)\in\mca^G$.}\label{berger-fig2}
\end{figure}

\switchcolumn

\begin{figure}
\begin{tikzpicture}[scale=0.9]
\draw[->,thick] (-1, 0) -- (5, 0) node[right] {$x$};
\draw[->,thick] (0, -4) -- (0, 2) node[above] {$y$};
\draw[ultra thin,color=gray] (-1.5,-4.5) grid (5.5,2.5); 

\draw[smooth, teal, very thick, samples=100, domain=0:5] plot(\x, {2*(0.5-0.5*sqrt(4*\x))});

\draw[-,dashed] (0,0) -- (4,-4);
\draw[very thick, fill, blue] (1,-1) circle (1pt) node[below left] {{\tiny $\ricci(g_{1/4})$}};
\draw[-,dashed] (0,0) -- (5,-3.75);
\draw[very thick, fill, red] (4,-3) circle (1pt) node[right] {{\tiny $\ricci(g_{1/2-b})$}};
\draw[-,dashed] (0,0) -- (1,2);
\draw[very thick, fill, red] (0.444,-0.333) circle (1pt) node[right] {{\tiny $\ricci(g_b)$}};
\draw[very thick, fill] (0,1) circle (0.5pt);
\draw[very thick, fill, violet] (0.25,0) circle (1pt) node[above right] {{\tiny $\ricci(g_{1/2})$}};
\draw[very thick, fill, orange] (0.062,0.5) circle (1pt) node[left] {{\tiny $\ricci(g_1)$}}; 

\draw (0,1) node[left] {{\tiny $\infty$}};
\draw (0,0) node[below left] {{\tiny $0$}};

\draw (4.5,1.5) node {$\sca^2(M)^G$};
\draw (1,-2.5) node[teal] {$\ricci(\mca^G)$};
\end{tikzpicture}
\caption{Ricci tensors of Berger spheres, $\ricci(g_b)=\left(\tfrac{1}{4b^2},\tfrac{2b-1}{4b},\tfrac{2b-1}{4b}\right)$, i.e., $y=\unm-\unm\sqrt{x}$, $0<x$.}\label{berger-fig1}
\end{figure}
\end{paracol}

It is easy to check that these are all the pairs $(g_b,g_{b'})$ having the same Ricci tensors up to scaling (see Figure \ref{berger-fig1}).  Interestingly enough, the twin metrics $g_b$ and $g_{1/2-b}$ with identical Ricci tensor up to scaling have scalar curvature of different sign.  

If we set $\overline{g}=-\kil_{\sug(2)}$ as a background metric, then $\dete_{\overline{g}}g_b=b^2$ and so by \eqref{berger1}, for any $0<b<\unm$,  
$$
\widetilde{\ricci}(g_b) = b^2\ricci(g_b) = \left(1/2-b\right)^2\ricci\left(g_{1/2-b}\right) = \widetilde{\ricci}\left(g_{1/2-b}\right).
$$
In particular, in any neighborhood of $\ricci(g_{1/4})$ there is a $T$ for which there exist two different constants $c_1,c_2$ and two different metrics $g_1,g_2$ near $g_{1/4}$ such that $\ricci(g_i)=c_iT$.  The image of the function $\widetilde{\ricci}:\mca^G\rightarrow\sca^2(M)^G$ has been drawn in Figure \ref{berger-fig3}.  
\end{example}

\subsection{Questions}\label{quest}
The following natural interrelated questions arise:

\begin{enumerate}[(Q1)]
\item {\it Uniqueness}.  Which metrics are determined by their Ricci tensors up to scaling? (i.e., $\ricci^{-1}(\ricci(g))=\RR_+g$).  Does every Einstein or Ricci positive metric satisfy that?    

Under what conditions on $T$ is the constant $c$ unique?  Does $\ricci(g)>0$ imply that $\ricci^{-1}(\RR_+\ricci(g))=\RR_+g$?  
Given $T$, what kind of set is 
$$
\{ c>0: \mbox{there exists solution to \eqref{PRP}}\}.
$$
Is it bounded below?  Is it finite?    
\item {\it Openness}.  Can the image of $\widetilde{\ricci}$ be open in $\sca^2(M)^G$?  Is $\widetilde{\ricci}(\mca^G)\cap\mca^G$ always open in $\mca^G$? 
\item {\it Local invertibility}.  At which metrics $g\in\mca^G$ is the function $\widetilde{\ricci}$ a local diffeomorphism?  Is this equivalent to have that $\Ker d\ricci|_g=\RR g$?
\item {\it Aut-isometry class}.  Since $\ricci(f^*g)=f^*\ricci(g)$ for any $f\in\Aut(G/K)$ (see Section \ref{equiv-sec}), it is natural to ask which metrics satisfy that the function $\ricci$ restricted to the equivariant isometry class of $g$, i.e., $\ricci:\Aut(G/K)\cdot g\longrightarrow\Aut(G/K)\cdot \ricci(g)$, is a (local) diffeomorphism.
\item {\it Signature}.  What are all the possible signatures of $\ricci(g)$ as $g$ runs through $\mca^G$? 
\end{enumerate}

\begin{figure}
\begin{tikzpicture}[scale=0.6]
\draw[->,thick] (-1, 0) -- (5, 0) node[right] {$x$};
\draw[->,thick] (0, -4) -- (0, 2) node[above] {$y$};
\draw[ultra thin,color=gray] (-1.5,-4.5) grid (5.5,2.5); 

\draw[smooth, teal, very thick, samples=100, domain=0:5] plot(\x, {2*(0.5-0.5*sqrt(4*\x))});

\draw[-,ultra thick, olive] (0,0) -- (4,-4);
\draw[very thick, fill, blue] (1,-1) circle (1pt) node[below left] {{\tiny $\ricci(g_{1/4})$}};

\draw (0,1) node[left] {{\tiny $\infty$}};
\draw (0,0) node[below left] {{\tiny $0$}};

\draw (-1,1.5) node {$\sca^2(M)^G$};
\draw (4,-1) node[olive] {$\widetilde{\ricci}(\mca^G)$};

\draw[fill=olive, opacity=0.4] (0,2.5) -- (0,0) -- (4,-4) -- (5.5,-4) -- (5.5,2.5) -- cycle;
\end{tikzpicture}
\caption{$\widetilde{\ricci}(\mca^G)=\RR_+\ricci(\mca^G)$ (i.e., \eqref{PRP}-solvable $T$'s)}\label{berger-fig3}
\end{figure}

Some remarks about these questions follow:

\begin{enumerate}[(R1)]
\item 
\begin{enumerate}[{\small $\bullet$}]
\item See \cite[Theorem 4.1]{Hml}, \cite[Corollary 3.3]{DtrKso}, \cite[Proposition 3.1]{Plm} and \cite[Lemma 4.6]{PlmRbn} for some known results on uniqueness.

\item Suppose that $g\in\mca^G$ is de Rham reducible in the following strong sense: $G=G_1\times G_2$, $K=K_1\times K_2$, $K_i\subset G_i$, $M=M_1\times M_2$ and $g=g_1+g_2$, where $g_i$ is a $G_i$-invariant metric on $M_i=G_i/K_i$.   Thus any metric of the form $c_1g_1+c_2g_2$ belongs to $\mca^G$ and
$$
\ricci(c_1g_1+c_2g_2) = \ricci(g_1)+\ricci(g_2) = \ricci(g), \qquad\forall c_1,c_2>0,
$$
giving rise to a non-uniqueness situation.  

\item We note that the product metric above can be Einstein or Ricci positive, so the questions about uniqueness only make sense if this product case is excluded.  

\item There is a curve of left-invariant metrics $g_t$ on $\SU(2)$ (resp.\ $\Sl_2(\RR)$) such that $\ricci(g_t)$ is constant and has signature $(+,0,0)$ (resp.\ $(+,-,-)$) (see \cite{Btt} and Example \ref{dim3} below).  
\end{enumerate}

\item 
\begin{enumerate}[{\small $\bullet$}]
\item If $G/K$ is isotropy irreducible, then $\mca^G=\RR_+g$, $\sca^2(M)^G=\RR g$ and the image of $\widetilde{\ricci}$ is $\RR_+g$, which is open in $\RR g$. 

\item The image of $\widetilde{\ricci}$ is neither closed nor open for $M=\SSS^3$ and $G=\SU(2)$ (see \cite{Btt}).  
\end{enumerate}

\item 
\begin{enumerate}[{\small $\bullet$}] 
\item $\dim{\Ker d\ricci|_g}\geq 2$ and $\widetilde{\ricci}$ is not a local diffeomorphism at any product metric $g=g_1+g_2$ as above.

\item Under the conditions as in the first item of (R1) above, it is easy to show that if the isotropy representation of $G_1/K_1$ does not contain any trivial subrepresentation, then any $G$-invariant metric on $M$ is necessarily a product metric.  In that case, $\widetilde{\ricci}$ is a local diffeomorphism at no point.  \end{enumerate}

\item Let $M=H$ be a Lie group and $K\subset H$ a closed subgroup.  Suppose that $g$ is a left-invariant metric on $H$ which is not $\Ad(K)$-invariant but has an $\Ad(K)$-invariant Ricci tensor $\ricci(g)$.  In other words, $\ricci(g)\in\sca^2(M)^G$ and $g\notin\mca^G$ for $G=H\times K$.  Thus there exist $k_t\in K$, $t\in (-\epsilon,\epsilon)$, such that $g_t:=\Ad(k_t)^*g$ is an injective curve with $\ricci(g_t)=\Ad(k_t)^*\ricci(g)\equiv\ricci(g)$.  Explicit examples of this phenomenon will be given in Example \ref{so5-1} for $H=\SO(5)$.  Question (Q4) will be studied in more detail in Section \ref{equiv-dRic}. 

\item If $\widetilde{\ricci}$ is a local diffeomorphism at $g$ and $\ricci(g)$ has signature $(s^-,s^0,s^+)$, then any signature of the form $(s^-+i,s^0-i-j,s^++j)$ with $i+j\leq s^0$ is attained among $\mca^G$.  An explicit application of this will be given in Example \ref{so5-3} for $M=\SO(5)$ and $G=\SO(5)\times T^2$.  We refer to \cite{ArrLfn} and the references therein for the study of Ricci signature on Lie groups.  
\end{enumerate}

\subsection{Variational principle}\label{VP-sec}
The manifold $\mca^G$ can be naturally endowed with a Riemannian metric, defined by 
\begin{equation}\label{metg}
\la T,S\ra_g := \tr{AB}, \quad \mbox{where} \quad T_o=g_o(A\cdot,\cdot), \quad S_o=g_o(B\cdot,\cdot), \quad\forall T,S\in\sca^2(M)^G.
\end{equation}
Equivalently, $\la T,T\ra_g:=\sum T(X_i,X_i)^2$ for any $g_o$-orthonormal basis $\{ X_i\}$ of $T_oM$.  Note that $\tr_g{T}=\tr{A}$ and $\ricci(g)\perp_g g$ if and only if $\scalar(g)=0$, where $\perp_g$ denotes orthogonality with respect to $\ip_g$.  It is well known (see, e.g., \cite{BhmWngZll, Nkn, Hbr}) that if $G$ is unimodular, then the gradient of the scalar curvature functional $\scalar:\mca^G\rightarrow\RR$, $\scalar(g):=\tr_g{\ricci(g)}$, is given by 
\begin{equation}\label{grad-sc}
\grad(\scalar)_g=-\ricci(g), \qquad\forall g\in\mca^G. 
\end{equation}
Since the tangent space of the submanifold $\mca^G_1:=\{ g'\in\mca^G:\dete_{\overline{g}}{g'}=1\}$ at a metric $g\in\mca^G_1$ is precisely $(\RR g)^{\perp_g}$, one obtains that $g\in\Crit(\scalar|_{\mca^G_1})$ if and only if $g$ is Einstein.  

On the other hand, for any fixed $T\in\sca^2(M)^G$, one has that $\ricci(g)=cT$ for some $c\in\RR$ if and only if $g$ is a critical point of $\scalar$ restricted to the submanifold 
$$
\mca^G_T:=\{ g'\in\mca^G:\tr_{g'}{T}=1\},
$$
as it easily follows that $T_{g}\mca^G_T=\{ T\}^{\perp_{g}}$ for any $g\in\mca^G_T$.  Note that if $g\in\mca^G_T$ and $\ricci(g)=cT$ then $\scalar(g)=c$.  Thus $g\in\Crit(\scalar|_{\mca^G_T})$ if and only if $(g,\pm\scalar(g))$ is a solution to the PRP \eqref{PRP} for $\pm T$.  It is important to note that there may be solutions with zero scalar curvature, which are precisely the critical points of $\scalar$ restricted to 
$$
\mca^G_{T,0}:=\{ g'\in\mca^G:\tr_{g'}{T}=0\}.  
$$
We refer to \cite{ArrGldPlm, ArrPlmZll, BttPlm} and the references therein for successful applications of this variational approach to the PRP on several different classes of homogeneous spaces.

\section{Ricci local invertibility}\label{RLI-sec}

We study in this section the following concept providing nice existence and uniqueness results for the PRP.    
Recall from the beginning of Section \ref{preric-sec} the functions $\ricci, \widetilde{\ricci}:\mca^G\rightarrow\sca^2(M)^G$.  

\begin{definition}\label{RLI-hom}
A metric $g_0\in\mca^G$ is said to be {\it Ricci locally invertible} if there exist an open neighborhood $U_1$ of $g_0$ in the submanifold 
$$
\mca_0^G:=\{ g\in\mca^G:\dete_{\overline{g}}{g}=\dete_{\overline{g}}{g_0}\}
$$ 
and an open neighborhood $V$ of $\ricci(g_0)$ in $\sca^2(M)^G$, such that the following conditions hold: 

\begin{enumerate}[{\rm (a)}]
\item $\ricci(U_1)$ is a submanifold of codimension one of $V$ and $\ricci:U_1\rightarrow\ricci(U_1)$ is a diffeomorphism; 

\item for any $T\in V$, there exists a unique pair $(g,c)$ with $g\in U_1$ and $c>0$ such that $\ricci(g)=cT$.   
\end{enumerate} 
\end{definition}

In other words, near $g_0$, the function $\ricci$ is as bijective as it can be, and for any $T$ near $\ricci(g_0)$, existence and (local) uniqueness of a solution to the PRP \eqref{PRP} hold.  
 
Some useful observations on the definition are in order:   

\begin{enumerate}[{\small $\bullet$}]
\item There may exist another constant $c'>0$ such that $\ricci(g')=c'T$ for some $g'\notin U_1$, as Example \ref{Berger} shows.  It follows from the last sentence in such example that the metric $g_{1/4}$ is not Ricci locally invertible, which also becomes clear from the image of the function $\widetilde{\ricci}$ given in Figure \ref{berger-fig3}. 

\item The subset 
\begin{equation}\label{Minv}
\mca^G_{inv}:=\left\{ g\in\mca^G:g\;\mbox{is Ricci locally invertible}\right\}
\end{equation}
is invariant under scaling since, for any $a>0$, conditions (a) and (b) both hold for the open neighborhoods $aU_1$ of $ag_0$ and $V$ of $\ricci(ag_0)=\ricci(g_0)$.  

\item $\mca^G_{inv}$ is open in $\mca^G$.  Indeed, any $g\in U_1$ is Ricci locally invertible (same $U_1$ and $V$) and so $\RR_+U_1\subset\mca^G_{inv}$, which can easily be shown to be open by using that $g\perp_gT_g\mca_0^G$ for any $g\in\mca_0^G$.  

\item The set $\RR_+\ricci(\mca^G_{inv})=\widetilde{\ricci}(\mca^G_{inv})$ is open in $\sca^2(M)^G$, as it follows from condition (b) that it contains the open subset $V$.  

\item For each $T\in\ricci(\mca^G_{inv})$, the subset $\ricci^{-1}(T)\cap\mca^G_{inv}\cap\mca^G_1$ is discrete (is it always finite?), where 
$$
\mca_1^G:=\{ g\in\mca^G:\dete_{\overline{g}}{g}=1\}.
$$
\item A natural question, to be studied below in this section, arises: is the open subset $\mca^G_{inv}$ dense in $\mca^G$?  Note that when that is the case, the open subset $\widetilde{\ricci}(\mca^G_{inv})$ is also dense in $\widetilde{\ricci}(\mca^G)$. 

\item The property of Ricci local invertibility is geometric, in the sense that $g\in\mca^G_{inv}$ if and only if each $f^*g$ is so for any $f\in\Aut(G/K)$ (see Section \ref{equiv-sec}). 

\item When $G$ is unimodular, for any $g\in U_1$, the tangent spaces $T_gU_1$ and $T_{\ricci(g)}\ricci(U_1)$ are both precisely the $\ip_{g}$-orthogonal complement of $\RR g$ (see Corollary \ref{imort} below).  In particular, 
\begin{equation}\label{dRcort}
\mca^G_{inv}\subset\mca^G_{\ricci}:=\left\{ g\in\mca^G:\Ker d\ricci|_g=\RR g\right\},  
\end{equation}
since $d\ricci|_g:(\RR g)^{\perp_g}\longrightarrow (\RR g)^{\perp_g}$ is an isomorphism for any $g\in U_1$ by condition (a).  Recall that $\RR g$ is always contained in $\Ker d\ricci|_g$. 
\end{enumerate}

We also consider the following subsets of metrics,  
\begin{equation}\label{Minvt}
\mca^G_{\widetilde{\ricci}}:=\left\{ g\in\mca^G:\widetilde{\ricci}\; \mbox{is a local diffeomorphism at}\; g\right\},  
\end{equation}
and
$$
\mca^G_{\scalar}:=\left\{ g\in\mca^G:\scalar(g)\ne 0\right\}.
$$
The following topological properties of the set $\mca^G_{\widetilde{\ricci}}$ are a consequence of the well-known fact that $\ricci(g)$ is a rational function in the coordinates $g_{ij}$'s of the metric $g\in\mca^G$. 

\begin{lemma}\label{RLI-lem}
Each of the subsets $\mca^G_{\ricci}$ and $\mca^G_{\widetilde{\ricci}}$ is either empty or open and dense in $\mca^G$, and $\mca^G_{\scalar}$ is always open and dense unless $G$ is abelian. 
\end{lemma}

\begin{remark}
The proof of the lemma shows that the complement of any of these three subsets is actually a real semi-algebraic subset of the vector space $\sca^2(M)^G$.  The following properties are classical theorems of Whitney (see e.g.\ \cite{BchCstRoy}): each of such complements has finitely many connected components, it admits a stratification (locally) into real algebraic submanifolds and it is locally path-connected, so path components and connected components coincide.
\end{remark}

\begin{proof}
We fix a basis $\{ T_1,\dots,T_m\}$ of $\sca^2(M)^G$ and for each $g=x_1T_1+\dots+x_mT_m\in\mca^G$, we write $\ricci(g)=R_1(g)T_1+\dots+R_m(g)T_m$.  In this way, $\ricci$ can be viewed as a differentiable function from an open subset of $\RR^m$ into $\RR^m$ with coordinates $R_i(x_1,\dots,x_m)$.  On the other hand, it is well known that the coordinates of $\ricci(g)$ are rational functions in the variables $g_{ij}$'s, where $[g_{ij}]$ is the matrix of $g\in\mca^G$ with respect to a fixed ordered basis $\{ X_1,\dots,X_n\}$ of $T_oM$ (see e.g. the proof of \cite[Proposition 1.5]{BhmWngZll}), and so each real function $R_i$ is also rational in $(x_1,\dots,x_m)$.  This implies that that each partial derivative $\frac{\partial}{\partial x_j} R_i$ is given by $P_{ij}/Q_{ij}$ for certain polynomials $P_{ij}$ and $Q_{ij}$ in $(x_1,\dots,x_m)$ such that $Q_{ij}|_{\mca^G}>0$ for all $i,j$.  Therefore,  
$$
\mca^G_{\ricci}=X\cap\mca^G, 
$$ 
where $X$ is the algebraic subset of the vector space $\sca^2(M)^G$ defined as the set of common zeroes of all the $(m-2)$-minors of the matrix $[P_{ij}/Q_{ij}]$ of $d\ricci|_g$.  It follows that $\mca^G_{\ricci}$ is either empty or dense in $\mca^G$.  

The cases of $\mca^G_{\widetilde{\ricci}}$ and $\mca^G_{\scalar}$ can be proved in much the same way as above.  If $\mca^G_{\scalar}$ is empty, i.e., $\scalar(g)=0$ for any $g\in\mca^G$, then $\ricci(g)=0$ for any $g\in\mca^G$ by \eqref{grad-sc} and so any $g\in\mca^G$ is flat (see \cite{AlkKml}).  It is easy to show that this is equivalent to $G$ abelian, concluding the proof.  
\end{proof}

As expected, the scalar curvature does play a role in the invertibility of Ricci. 

\begin{theorem}\label{RLI-thm}
Let $M=G/K$ be a homogeneous space.   

\begin{enumerate}[{\rm (i)}]
\item $\mca^G_{\widetilde{\ricci}}\subset \mca^G_{inv}$. 

\item For $G$ unimodular, we have that $\mca^G_{\widetilde{\ricci}}=\mca^G_{\ricci}\cap\mca^G_{\scalar}$.  

\item If $G$ is unimodular, then $\mca^G_{inv}\cap\mca^G_{\scalar}\subset\mca^G_{\widetilde{\ricci}}$.  

\item If $G$ is nonabelian and $\mca^G_{\ricci}$ is nonempty, then $\mca^G_{inv}$ and $\mca^G_{\widetilde{\ricci}}$ are so.  
\end{enumerate}
\end{theorem}

\begin{remark}
We do not know whether the equality in part (i) holds. 
\end{remark}

\begin{proof}
We first prove part (i).  Suppose that $\widetilde{\ricci}$ is a local diffeomorphism at $g_0\in\mca^G$, which can be assumed to have $\dete_{\overline{g}}{g_0}=1$.  Thus there exist open neighborhoods $\widetilde{U}$ and $\widetilde{V}$ of $g_0$ and $\ricci(g_0)$ in $\mca^G$ and $\sca^2(M)^G$, respectively, such that $\widetilde{\ricci}:\widetilde{U}\rightarrow \widetilde{V}$ is a diffeomorphism.  By using that $\sca^2(M)^G$ is endowed with the Euclidean topology, it is easy to see that there exist $\epsilon>0$ and a subset $U_1\subset\widetilde{U}\cap\mca^G_1$, open in $\mca^G_1$, such that the open neighborhood $(1-\epsilon,1+\epsilon)U_1$ of $g_0$ is contained in $\widetilde{U}$ (the openness of $(1-\epsilon,1+\epsilon)U_1$ easily follows from the fact that $g\perp_gT_g\mca_1^G$ for any $g\in\mca_1^G$).  Since $\widetilde{\ricci}:(1-\epsilon,1+\epsilon)U_1\rightarrow V$ is also a diffeomorphism, where 
$$
V:=\widetilde{\ricci}((1-\epsilon,1+\epsilon)U_1)=((1-\epsilon)^n,(1+\epsilon)^n)\ricci(U_1),  
$$  
we obtain from $\widetilde{\ricci}|_{U_1}=\ricci$ that part (a) of Definition \ref{RLI-hom} holds for $U_1$ and the open neighborhood $V$ of $\ricci(g_0)$.  The existence condition in part (b) also holds: given $T\in V$, there exists $c\in ((1-\epsilon)^n,(1+\epsilon)^n)$ such that $cT\in\ricci(U_1)$, that is, $\ricci(g)=cT$ for some $g\in U_1$.  

In order for the uniqueness in part (b) to hold, one may have to consider smaller neighborhoods $U_1$ and $V$ with the above properties.  If we assume that, on the contrary, such a choice is impossible, then there exist sequences $g_k,g'_k\in\mca_1^G$ converging to $g_0$, as $k\to\infty$, such that $g'_k\ne g_k$ and $\ricci(g'_k)=c_k\ricci(g_k)$ for some $c_k>0$, for any $k\in\NN$.  This implies that $c_k\to 1$ and 
$$
\widetilde{\ricci}(g'_k) = \ricci(g'_k) = c_k\ricci(g_k)  
= \widetilde{\ricci}(a_kg_k), \qquad\forall k\in\NN,
$$
where $a_k:=c_k^{1/n}\to 1$ and therefore by the injectivity of $\widetilde{\ricci}$ near $g_0$, one obtains that $g'_k=a_kg_k$ for sufficiently large $k$, a contradiction.  This concludes the proof of part (i).    

In order to prove part (ii), we first assume that $g\in\mca^G_{\widetilde{\ricci}}$, i.e., $d\widetilde{\ricci}|_g$ is an isomorphism.  Since 
\begin{equation}\label{dRicdet}
d\widetilde{\ricci}|_g T = (\tr_{\overline{g}}{T})\ricci(g) + (\dete_{\overline{g}}{g})d\ricci|_gT, \qquad\forall T\in\sca^2(M)^G,
\end{equation}
we obtain that
$$
\Ker d\ricci|_g \cap\left\{ T\in\sca^2(M)^G:\tr_{\overline{g}}{T}=0\right\}=0.
$$   
Thus $\Ker d\ricci|_g=\RR g$.  Indeed, if $d\ricci|_gT=0$, then for $a:=-\tr_{\overline{g}}{T}/\tr_{\overline{g}}{g}$ one has that $d\widetilde{\ricci}|_g (ag+T)=0$ and so $T=-ag$.  On the other hand,  \eqref{dRicdet} also implies that $\scalar_g\ne 0$, since otherwise $\ricci(g)\perp g$ and so there exists a nonzero $T\in\sca^2(M)^G$ such that $d\ricci|_g T=-(\dete_{\overline{g}}{g})^{-1}\ricci(g)$ (see \eqref{dRcort}), yielding $d\widetilde{\ricci}|_g (ag+T)=0$ for $a:=(1-\tr_{\overline{g}}{T})/\tr_{\overline{g}}{g}$, a contradiction.  

Conversely, if $\Ker d\ricci|_g=\RR g$ and $\scalar_g\ne 0$, then the fact that $d\widetilde{\ricci}|_g T=0$ implies that $(\tr_{\overline{g}}{T})\scalar(g)=0$ (recall that $d\ricci|_gT\perp g$).  Thus $\tr_{\overline{g}}{T}=0$ and $d\ricci|_gT=0$ by \eqref{dRicdet}, from which follows that $T\in\RR g$ and so $T=0$, concluding the proof of part (ii).   

Finally, we note that part (iii) follows from part (ii) and \eqref{dRcort}, and part (iv) follows from Lemma \ref{RLI-lem} and part (ii), concluding the proof.  
\end{proof}

In the case when $G$ is compact, it is well known that there is always a metric $g\in\mca^G$ such that $\scalar(g)=0$, unless $M=G/K$ is covered by a product of isotropy irreducible homogeneous spaces, in which case the function $\ricci$ is constant (see \cite[Theorem 2.1]{WngZll}).  A first corollary of Theorem \ref{RLI-thm} therefore follows (cf. Example \ref{heis} below).  

\begin{corollary}\label{RLI-cor2}
For $G$ compact, the function $\widetilde{\ricci}:\mca^G\rightarrow\sca^2(M)^G$ is never a local diffeomorphism.  
\end{corollary}

The sufficient condition for Ricci locally invertibility provided by Theorem \ref{RLI-thm}, (i) paves the way to prove the density of the open set $\mca^G_{inv}$.  

\begin{corollary}\label{RLI-cor1}
If $G$ is unimodular, then the open subset $\mca^G_{inv}$ is either empty or dense in $\mca^G$.  
\end{corollary}

\begin{remark}
In particular, it is easy to check that if nonempty, the subsets $\ricci(\mca^G_{inv})$ and $\widetilde{\ricci}(\mca^G_{inv})$ are open and dense in the images $\ricci(\mca^G)$ and $\widetilde{\ricci}(\mca^G)$, respectively. 
\end{remark}

\begin{proof}
If $\mca^G_{inv}$ is nonempty, then $\mca^G_{\widetilde{\ricci}}$ is nonempty by Theorem \ref{RLI-thm}, (iii).  Indeed, if $\scalar(g)=0$ for every $g$ in the open set $\mca^G_{inv}$, then $\ricci(g)=0$ for any $g\in\mca^G_{inv}$ by \eqref{grad-sc}, which is a contradiction.  It therefore follows from Lemma \ref{RLI-lem} and Theorem \ref{RLI-thm}, (i)  that $\mca^G_{inv}$ is either empty or dense in $\mca^G$, as was to be shown.   
\end{proof}

A new natural question arises: what is most likely, either $\mca^G_{inv}$ empty or open and dense?  The following is a summary of what we know at this stage about this problem.   

\begin{enumerate}[{\small $\bullet$}]
\item If $M=G_1/K_1\times G_2/K_2$, $G=G_1\times G_2$ and the isotropy representation of $G_1/K_1$ does not contain any trivial subrepresentation, then $\mca^G_{inv}$ is empty (see the second item of remark (R3) above).

\item The above condition is not necessary: it is shown in Example \ref{s5s1} below that $\mca^G_{inv}$ is empty for  $M=\SSS^5\times\SSS^1$ and $G=\SU(3)\times S^1$.   

\item However, $\mca^G_{inv}$ is open and dense in the following cases:  $M=\SSS^5\times\SSS^5$ and $G=\SU(3)\times \SU(3)$ (see Example \ref{s5s5}), $M=\SSS^7\times\SSS^5$ and $G=\SU(4)\times \SU(3)$ (see Example \ref{s7s5}), as well as in the Lie group case $M=\SSS^3\times\SSS^1$ and $G=\SU(2)\times S^1$ (see Example \ref{s3s1}).   

\item On the other hand, we shall prove in Section \ref{natred-sec} that for any compact homogeneous space $M=G/K$ such that $\kg$ is $\ggo$-indescomposable (see Definition \ref{dec-def}), $\mca^G_{inv}$ is open and dense.  Note that this condition is not necessary, as the examples in the above item show.  

\item We also show in Section \ref{natred-sec} that $\mca^G_{inv}\ne\emptyset$ for any $M=G/K$ admitting an holonomy irreducible naturally reductive metric, which includes the non-compact examples given in Section \ref{exa-nr}.  

\item $\mca^G_{inv}$ is open and dense in the huge space $\mca^G$ of all left-invariant metrics on a (not necessarily compact) simple Lie group $M=G$, depending on $n(n+1)/2$ parameters (see Theorem \ref{simple-thm}).  

\item Consider $M=G$, where $G$ is a simply connected $2$-step nilpotent Lie group such that $[\ggo,\ggo]\ne\zg(\ggo)$, i.e., $\ggo=\ggo_1\oplus\ag$ for some ideal $\ggo_1$ and abelian ideal $\ag$.  Using that any $D\in\glg(\ggo)$ such that $D[\ggo_1,\ggo_1]=0$ and $D\ggo\subset [\ggo_1,\ggo_1]\oplus\ag$ is a derivation, it is easy to prove that any left-invariant metric on $G$ is in the $\Aut(G)$-orbit of a metric satisfying $\ggo_1\perp\ag$, which can never be Ricci locally invertible.  Thus $\mca^G_{inv}$ is empty.  

\item Beyond the above family and the abelian case, we do not know of any other Lie group $M=G$ satisfying that $\mca^G_{inv}=\emptyset$.  
\end{enumerate}

\begin{remark}
A complete characterization of homogeneous spaces with $\mca^G_{inv}$ empty would be desirable, we leave it as a  natural open problem.  
\end{remark}

\section{Examples}

In this section, we develop many examples illustrating some aspects of Ricci locally invertibility and the questions and remarks made in Section \ref{quest}.  

\begin{example}\label{dim3} {\it Dimension $3$}.  
Let $\ggo$ be the $3$-dimensional Lie algebra with Lie bracket defined by
$$ 
[X_1,X_2]=-\tfrac{1}{\sqrt{2}} X_3,\quad [X_1,X_3]=-\tfrac{1}{\sqrt{2}} \epsilon X_2, \quad [X_2,X_3]=-\tfrac{1}{\sqrt{2}} X_1, \qquad\epsilon=\pm 1.
$$
Since $\ggo \simeq \sug(2)$ for $\epsilon=-1$ and $\ggo \simeq \slg_2(\RR)$ for $\epsilon=1$, we are covering in this example the cases of all left-invariant metrics on either $M=\SU(2)=G$ or $M=\Sl_2(\RR)=G$, depending on whether $\epsilon =\pm 1$.   By writing any symmetric $2$-form with respect to the ordered basis $\{ X_1,X_2,X_3\}$, a straightforward computation gives that the Ricci tensor of the metric $g_{a,b,d}:=(a,b,d)$ equals (see \cite{Mln} or \cite{Btt}):
\begin{equation}\label{ric-dim3}
\ricci(g_{a,b,d})=\left(\tfrac{a^2-(d + \epsilon b)^2}{4 b d},\tfrac{b^2-(a-d)^2}{4 a d},\tfrac{d^2-(a + \epsilon b)^2}{4 a b}\right).
\end{equation}
It was proved in \cite[Theorem 6.1]{Hml} that in the case $M=\SU(2)$ (i.e., $\epsilon=-1$), for any positive $T=(T_1,T_2,T_3)$, there exists a unique metric $g_{a,b,d}$ and a unique $c>0$ such that $\ricci(g)=cT$.  More recently, in \cite{Btt}, the PRP was completely solved for all unimodular $3$-dimensional Lie groups via this approach (there is a mistake in \cite[Theorem 4.1, (iii)]{Btt}, due to the fact that condition in \cite[Lemma 4.4, (ii)]{Btt} never holds).  

It follows from \eqref{ric-dim3} that none of the following metrics is Ricci locally invertible: 

\begin{enumerate}[{\small $\bullet$}]
\item If $\epsilon=-1$, then $\ricci(g_{1,b,1-b})= (1,0,0)$ for any $0<b<1$.  

\item If $\epsilon=1$, then $\ricci(g_{1,1+d,d})= (-1,1,-1)$ for any $0<d$.  
\end{enumerate}
\end{example}

\begin{example}\label{heis}
If $\ggo$ is the $3$-dimensional Heisenberg Lie algebra, i.e., 
$ 
[X_1,X_2]=X_3,  
$
then the Ricci tensor of any left-invariant metric $g_{a,b,d}:=(a,b,d)$ on $M=G$, with respect to the ordered basis $\{ X_1,X_2,X_3\}$, is given by 
$$
\ricci(g_{a,b,d})=\left(-\tfrac{d}{2b},-\tfrac{d}{2a},\tfrac{d^2}{2ab}\right).
$$
Since $\RR_+\Aut(\ggo)\cdot\{ g_{1,1,1}\}=\mca^G$, one easily obtains that $\widetilde{\ricci}:\mca^G\rightarrow \sca^2(M)_{0}$ is a diffeomorphism, where $\sca^2(M)_{0}$ denotes the subset of all symmetric $2$-forms of signature $(-,-,+)$. 
\end{example}

The following example may be viewed as a massive generalization of Berger metrics in Example \ref{Berger}.   

\begin{example}\label{DZ-PRP} {\it D'Atri-Ziller metrics}. 
We consider $\mca^G$ for a compact Lie group $M=H=G/\Delta K$, where $G=H\times K$ and $K\subset H$ is a proper closed Lie subgroup.  Thus $\mca^G$ is identified with the space of all left-invariant metrics on $H$ which are also $K$-invariant.  As a background metric we take any bi-invariant metric $\ip$ on $H$.  Consider an $\ad{\kg}$-invariant and $\ip$-orthogonal decomposition, 
$$
\hg = \ag \oplus \kg_1 \oplus \dots \oplus \kg_r,
$$
where each $\kg_i$ is either a simple ideal of $\kg$ or a one-dimensional subspace of $\zg(\kg)$ and assume that $\ag$ is $\ad{\kg}$-irreducible.  It follows from \cite[p.\ 33]{DtrZll} (see also \cite[Proposition 2.2]{ArrPlmZll}) that for any metric $g=a \ip_\ag + \sum b_i \ip_{\kg_i}\in \mca^G$, $a,b_i>0$, 
$$
\ricci(g) = \left(\tfrac{1}{4} - \sum \left(\tfrac{b_i}{a}-1\right)\tfrac{\dim{\kg_i}(1-k_i)}{2\dim{\ag}}\right) \ip_{\ag} + \tfrac{1}{4}\sum\left(k_i+(1-k_i)\tfrac{b_i^2}{a^2}\right) \ip_{\kg_i},    
$$
where the Killing form of $\kg_i$ is written as $\kil_{\kg_i} =k_i \kil_{\hg} |_{\kg_i\times\kg_i}$, $0\leq k_i\leq 1$ ($k_i=1$ if and only if $\kg_i$ is an ideal of $\hg$ by \cite[Theorem 11]{DtrZll}).  This implies that for a given $T=T_\ag \ip_{\ag} + \sum T_i \ip_{\kg_i}\in \sca^2(M)^G$, there exists a solution to the PRP \eqref{PRP1}, $\ricci(g)=T$, if and only if 
$$
T_\ag=\tfrac{1}{4} - \displaystyle{\sum_{i=1}^r}\tfrac{\dim{\kg_i}}{2\dim{\ag}}\left(\sqrt{(4T_i-k_i)(1-k_i)}+k_i-1\right) \quad\mbox{and}\quad T_i >\unc k_i,\quad\forall 1 \le i \le r,
$$
since $\frac{b_i}{a} = \sqrt{\frac{4T_i-k_i}{1-k_i}}$ for all $i$.  Note that the solution $g$ is always unique up to scaling (see \cite[Theorem 4.1]{ArrGldPlm} for the case of non-compact simple $H$).  We shall prove in Section \ref{natred-sec} that if $H$ is simple (or more in general, if $\kg$ is $\hg$-indecomposable as in Definition \ref{dec-def} below), then any metric $g\in\mca^G$ with nonzero scalar curvature is Ricci locally invertible.  

On the other hand, recall from Example \ref{Berger} that for solutions to the PRP \eqref{PRP}, $\ricci(g)=cT$, the constant $c$ may not be unique in this case.  The set $\RR_+\ricci(\mca^G)$ has been studied in \cite{ArrPlmZll} via the variational approach described in Section \ref{VP-sec}, as well as in \cite[Theorem 4.1]{ArrGldPlm} in the case of a non-compact simple Lie group $H$. 
\end{example}

\begin{example}\label{so5-1}
In order to study left-invariant metrics on $M=\SO(5)$, we fix the basis $\{ X_{ij}:=\frac{1}{\sqrt{6}}(E_{ij}-E_{ji}),\;1\leq i<j\leq 5\}$ of $\sog(5)$ endowed with the lexicographical order, which is orthonormal with respect to $-\kil_{\sog(5)}$.  It easily follows from \eqref{mm2} that, relative to $\{ X_{ij}\}$, $\ricci(g)$ is diagonal for every $g$ diagonal.  We consider $G=\SO(5)\times K$ for $K:=\exp{\RR X_{12}}\simeq S^1$ as in Example \ref{DZ-PRP}, so $\mca^G$ consists of those left-invariant metrics which are in addition $K$-invariant.  It is easy to check that a diagonal metric $g=(a_1,\dots, a_{10})$ with respect to $\{ X_{ij}\}$ is $K$-invariant if and only if $a_2=a_5$, $a_3=a_6$ and $a_4=a_7$.  For $g=\left(8+\frac{1}{2} \sqrt{292},1,1,2,2,2,1,64,1,1\right)\notin\mca^G$, a straightforward computation gives that
$$
\ricci(g)=\Diag(17+2\sqrt{73}, r,r,s,r,r,s,12, -\tfrac{1}{2},-\tfrac{1}{2}), \quad 
  r:=-\tfrac{17+2\sqrt{73}}{2(8+\sqrt{73})}, \quad s:=-\tfrac{9+\sqrt{73}}{2(8+\sqrt{73})}.
$$
Note that $\ricci(g)$ is $K$-invariant but $g$ is not.  This provides the circle $g_t:=\Ad(\exp{t X_{12}})^*g$ of left-invariant metrics on $\SO(5)$ with constant Ricci tensor, $\ricci(g_t)\equiv\ricci(g)$ (see Question (Q4) and the corresponding remark in Section \ref{quest}).  Other examples of non-K-invariant metrics with $K$-invariant Ricci tensors are 
$$  
g_1=\left(24+\sqrt{601},1,1,4,4,4,1,24,1,1\right), \quad 
g_2=\left(\tfrac{1}{4},1,1,1,\tfrac{3}{2},\tfrac{3}{2},\tfrac{1}{2},\tfrac{117}{8},u,u\right),
$$
where $u:=\tfrac{105 + \sqrt{12049}}{16}$.  Note that, in particular, none of these metrics is Ricci locally invertible in $\mca^{\SO(5)}$, i.e., as a left-invariant metric.  The metric $g_2$ in addition satisfies that $\ricci(g_2)(X_{12},X_{12})<0$, which implies that $\ricci(g_2)$ is not the Ricci tensor of any diagonal metric in $\mca^G$.  Indeed, it is easy to check that the Ricci tensor is positive at $X_{12}$ for any such metric.   
\end{example}

\begin{example}\label{so5-3}
Keeping the notation as in the above example, we now consider the maximal torus $K$ of $\SO(5)$ with Lie algebra  $\kg=\RR X_{12} + \RR X_{34}$.  It is easy to see that any metric of the form $g=\left(a_1,a_2,a_2,a_4,a_2,a_2,a_4,a_8,a_9,a_9\right)$ is $K$-invariant, i.e., $g\in\mca^G$, where $G=H\times K$.  For the numbers 
$$ 
a_1=8-\tfrac{1}{18} r^3-\tfrac{11}{18} r, \quad a_2=4,\quad a_4=1,\quad a_8= \tfrac{1}{2} r (r^2-5),\quad a_9=r,
$$
where $r$ is the positive root of $z^4+2 z^2 -36 z -135$ ($r\approx 3.96$, $a_1\approx 2.12$, $a_8\approx 21.17$), a straightforward computation gives that the corresponding metric $g$ has Ricci tensor
$$
\ricci(g) = \left(R_1,0,0,0,0,0,0,R_8,0,0\right), 
$$
where
$$
R_1:= -\tfrac{7}{96}r^3+\tfrac{1}{16}r^2-\tfrac{17}{24}r+\tfrac{217}{32}\approx 0.42,\qquad R_8:= \tfrac{3}{32}r^3+\tfrac{11}{48}r^2-\tfrac{3}{8}r-\tfrac{85}{96}\approx 7.05.
$$
Since $g$ is Ricci locally invertible (see Section \ref{natred-sec} below), we obtain that any signature of the form $(i,8-i-j,2+j)$, $0\leq i,j$, $i+j\leq 8$ is attained by some metric in $\mca^G$ (cf.\ (Q5) in Section \ref{quest}).  
\end{example}

\section{Moving bracket approach to PRP}\label{mba2-sec}

Let us fix a reductive decomposition $\ggo=\kg\oplus\pg$ for the homogeneous space $M=G/K$ and a background metric $g\in\mca^G$ with $g_o=\ip\in\sym_+^2(\pg)^K$, as in Section \ref{preli}.  The Lie bracket of $\ggo$ will always be denoted by $\mu$.   

The PRP \eqref{PRP1} (resp.\ \eqref{PRP}) is therefore equivalent to study the injectivity and the image (resp.\ the image up to scaling) of the function
\begin{equation}\label{Rc-MBA}
\ricci:\sym_+(\pg)^K\longrightarrow\sym^2(\pg)^K, \qquad \ricci(h):=\ricci(g_h), \quad g_h:=\la h\cdot,h\cdot\ra\in\mca^G.
\end{equation}
From the moving-bracket approach perspective described in Section \ref{mba-sec}, we first note that by  \eqref{ric2}, the function $\overline{\Ricci}:\sym_+(\pg)^K\longrightarrow\sym(\pg)^K$ given at each $h\in\sym_+(\pg)^K$ by 
\begin{align}
\overline{\Ricci}(h):=h\Ricci_{h\cdot\mu}h =& h\Mm_{h\cdot\mu_\pg}h - \unm\kil_\mu \label{Ricb}\\ 
& - \unm h^2\ad_\pg{h^{-2}H_{\mu_\pg}} - \unm (\ad_\pg{h^{-2}H_{\mu_\pg}})^th^2, \notag 
\end{align}
satisfies that 
$$
\ricci(h)=\la\overline{\Ricci}(h)\cdot,\cdot\ra, \qquad\forall h\in\sym_+(\pg)^K.
$$
Note that $\overline{\Ricci}(ah)=\overline{\Ricci}(h)$ for any $h\in\sym_+(\pg)^K$ and $a>0$.  The PRP \eqref{PRP1} therefore becomes $\overline{\Ricci}(h)=T$ for a given operator $T\in\sym(\pg)^K$, and the PRP \eqref{PRP} can be stated as follows: given an operator $T\in\sym(\pg)^K$, are there an $h\in\sym_+(\pg)^K$ and a constant $c>0$ such that $\overline{\Ricci}(h) = cT$?

\subsection{First variation of Ricci and the Lichnerowicz Laplacian}
We next compute the derivative of the Ricci curvature functions at the metric $g=\ip=g_I\in\mca^G$.  Note that after all the identifications and different functions considered, one has that 
\begin{equation}\label{ident-dRc}
d\ricci|_gT = \unm  d\ricci|_I A =  \la \unm d\overline{\Ricci}|_I A\cdot,\cdot\ra, 
\end{equation}
for any $T=\la A\cdot,\cdot\ra\in \sym^2(\pg)^K \equiv \sca^2(M)^G$, where $A\in\sym(\pg)^K$. 

\begin{lemma}\label{dRicmu}
For any $A\in\sym(\pg)^K$, 
\begin{align*}
d\overline{\Ricci}|_I A=& \unm S\circ\delta_{\mu_{\pg}}^t\delta_{\mu_{\pg}}(A)+A\Mm_{\mu_\pg} + \Mm_{\mu_\pg}A \\ 
& - A\ad_\pg{H_{\mu_\pg}} - (\ad_\pg{H_{\mu_\pg}})^tA + 2S\left(\ad_\pg{AH_{\mu_\pg}}\right).  
\end{align*}
\end{lemma}

\begin{remark}
Varying the background Lie bracket, we obtain from \eqref{Rc-MBA} that at each $h\in\sym_+(\pg)^K$,\begin{align*}
d\overline{\Ricci}|_h A=& \unm S\circ\delta_{h\cdot\mu_{\pg}}^t\delta_{h\cdot\mu_{\pg}}(Ah^{-1})+A\Mm_{h\cdot\mu_\pg} + \Mm_{h\cdot\mu_\pg}A \\ 
& - Ah\ad_\pg{(h^{-2}H_{\mu_\pg})}h^{-1} - (h\ad_\pg{(h^{-2}H_{\mu_\pg})}h^{-1})^tA \\ 
&+ 2S\left(h\ad_\pg{(h^{-1}Ah^{-1}H_{\mu_\pg})}h^{-1}\right),  
\end{align*}
for any $A\in\sym(\pg)^K$.
\end{remark}

\begin{remark}
It follows from the formula in the lemma that if $G$ is unimodular, i.e., $H_{\mu_\pg}=0$, then $d\overline{\Ricci}|_I$ is self-adjoint as an operator of $\sym(\pg)^K$, that is, $d\ricci|_g$ is a self-adjoint operator of $\sca^2(M)^G$ relative to $\ip_g$ (see \eqref{ident-dRc}).  
\end{remark}

\begin{remark}
For $G$ unimodular, the formula is equivalent to 
$$
\la d\ricci|_gT,T\ra_g = \unc \left|\theta(A)\mu_{\pg}\right|^2 + \tr{\Mm_{\mu_\pg} A^2}, \qquad\forall T=\la A\cdot,\cdot\ra\in\sca^2(M)^G.      
$$
Recall that $\tr{\Mm_{\mu_\pg} A^2} = \unc\la\theta(A^2)\mu_{\pg},\mu_{\pg}\ra$.  
\end{remark}

\begin{proof}
For any $A\in\sym(\pg)^K$, if $h(t):=I+tA$, then
\begin{align*}
d\overline{\Ricci}|_I A =& \left.\ddt\right|_0 \overline{\Ricci}(h(t)) 
= \left.\ddt\right|_0 h(t))\Mm_{h(t)\cdot\mu_\pg}h(t) - h(t)S(\ad_\pg{h(t)^{-1}H_{\mu_\pg}})h(t) \\ 
=& A\Mm_{\mu_\pg} + \Mm_{\mu_\pg}A + d\overline{\Mm}|_I A \\
&- A\ad_\pg{H_{\mu_\pg}} + \ad_\pg{AH_{\mu_\pg}} - (\ad_\pg{H_{\mu_\pg}})^tA + (\ad_\pg{AH_{\mu_\pg}})^t,
\end{align*}
and so the formula follows from Lemma \ref{dM}.  
\end{proof}

Due to the scaling invariance of $\overline{\Ricci}$, $\RR I$ is always contained in $\Ker d\overline{\Ricci}|_I$.  

\begin{corollary}\label{imort}
If $G$ is unimodular, then the image of $d\overline{\Ricci}|_I$ is orthogonal to $\RR I$.  
\end{corollary}

\begin{remark}
Equivalently, $d\ricci|_g\sca^2(M)^G\perp_g\RR g$ for any $g\in\mca^G$.  In the non-unimodular case, it is easy to see that $d\overline{\Ricci}|_IA\perp\RR I$ if and only if $\la A,\ad_\pg{H_{\mu_\pg}}\ra=\la AH_{\mu_\pg},H_{\mu_\pg}\ra$.  
\end{remark}

\begin{proof}
It follows from \eqref{mm1} that 
$$
\la d\overline{\Ricci}|_I A,I\ra = \la\delta_{\mu_{\pg}}(A),\delta_{\mu_\pg}(I)\ra + 2\la\Mm_{\mu_\pg}, A\ra = -\la\theta(A)\mu_\pg,\mu_\pg\ra + 2\la\Mm_{\mu_\pg}, A\ra = 0,
$$
and so the statement follows.  
\end{proof}

In the case when $G$ is compact, restricted to the subspace $\Ker\delta_g$ of divergence-free symmetric $2$-tensors, $d\ricci|_g=\unm\Delta_L$, where $\Delta_L$ is the {\it Lichnerowicz Laplacian} of $g$ (see \cite[Theorem 1.174, (d)]{Bss}).  Here $\delta_g:\sca^2(M)\rightarrow\Omega^1(M)$ is the divergence operator $\delta_g(T):=-\sum \nabla_{X_i}T(X_i,\cdot)$, where $\{ X_i\}$ is any local orthonormal frame.  The following formula for $\Delta_L$ therefore follows from Lemma \ref{dRicmu} and \eqref{ident-dRc}. 

\begin{corollary}\label{LL}
Let $M=G/K$ be a homogeneous space with $G$ compact, endowed with a reductive decomposition $\ggo=\kg\oplus\pg$.  Then the Lichnerowicz Laplacian of any $G$-invariant Riemannian metric $g$ on $M$ is given by
$$
\Delta_LT =  \unm S\circ\delta_{\mu_{\pg}}^t\delta_{\mu_{\pg}}(A)+A\Mm_{\mu_\pg} + \Mm_{\mu_\pg}A, 
$$
for any divergence-free $T=\la A\cdot,\cdot\ra\in \sym^2(\pg)^K \equiv \sca^2(M)^G$, where $A\in\sym(\pg)^K$ and $\ip := g_o\in\sym_+^2(\pg)^K$.
\end{corollary}

It is proved in \cite[Lemma 2.2]{WngWng} that if $G$ is compact, then $\delta_g(T)=0$ for any $g\in\mca^G$ and $T\in\sca^2(M)^G$ such that $T\circ g=g\circ T$ relative to a normal metric.

\subsection{Semisimple Lie groups}\label{simple-sec}
The formula in Lemma \ref{dRicmu} becomes quite simpler in the following case.  Let $M=G$ be a semisimple Lie group, so $\mca^G$ is the manifold of dimension $n(n+1)/2$ of all left-invariant metrics on $G$.  Consider a Cartan decomposition $\ggo=\hg\oplus\qg$ and the metric $g_{\kil}\in\mca^G$ defined by $\ip:=-\kil + \kil\in\sym_+(\ggo)$, where $\kil$ denotes the Killing form of $\ggo$.  Note that $G$ is compact if and only if $\qg=0$.  

Up to scaling, we can assume that $|\mu|^2=n$.  It is easy to see that therefore, $\Mm_\mu=-\unc I$, $\kil_\mu|_{\hg}=-I$, $\kil_\mu|_{\qg}=I$ and so $\Ricci_\mu|_{\hg}=\unc I$ and $\Ricci_\mu|_{\qg}=-\frac{3}{4}I$.  By fixing a $\ip$-orthonormal basis $\{ X_i\}$ of $\ggo$ such that each $X_i$ is either in $\hg$ or $\qg$ for all $i$, one obtains that $\ad{X_i}$ is either skew-symmetric or symmetric, respectively.  This implies that $|\ad{X}|=|X|$ for any $X\in\ggo$.  

We consider the Casimir operator $\cas_\ggo$ acting on the representation $\glg(\ggo)$ of $\ggo$ defined by $\tau(X)A:=[\ad{X},A]$, i.e., $\cas_\ggo=\sum\tau(X_i)^t\tau(X_i)$.  Thus $\cas_\ggo:\glg(\ggo)\longrightarrow\glg(\ggo)$ is given by
\begin{equation}\label{Cg-def}
\cas_\ggo(A):=\sum[(\ad{X_i})^t,[\ad{X_i},A]], \qquad\forall A\in\glg(\ggo).  
\end{equation}
Note that $\cas_\ggo\geq 0$, $\cas_\ggo(\sym(\ggo))\subset\sym(\ggo)$ and $\cas_\ggo(A)=0$ if and only if $[A,\ad{\ggo}]=0$.

\begin{lemma}\label{dRc-simple}
$d\overline{\Ricci}|_I\ad{X}=-\unm\ad{X}$ for any $X\in\qg$ and 
$$
d\overline{\Ricci}|_IA = \unm \cas_\ggo(A), \qquad\forall A\perp\ad{\qg}, \quad A\in\sym(\ggo). 
$$ 
\end{lemma}

\begin{proof}
By Lemma \ref{dRicmu} and the fact that $\Mm_\mu=-\unc I$, we have that $d\overline{\Ricci}|_I=\unm\delta_\mu^t\delta_\mu-\unm\id$, so the first statement of the lemma follows.  Now for any symmetric $A\perp\ad{\qg}$, it follows from \eqref{Cg-def} that, 
\begin{align*}
\la d\overline{\Ricci}|_IA,A\ra + \unm|A|^2 =&   
\unm |\delta_{\mu}(A)|^2 =
\unm\sum |\ad_{\delta_{\mu}(A)}{X_i}|^2 =  \unm\sum |\ad{AX_i} + [\ad{X_i},A]|^2 \\
=&  \unm\sum |\ad{AX_i}|^2 + |[\ad{X_i},A]|^2 +2\la\ad{AX_i},[\ad{X_i},A]\ra \\
=& \unm\sum |AX_i|^2 + \unm\la\cas_\ggo(A),A\ra +\sum \la[(\ad{X_i})^t,\ad{AX_i}],A\ra \\ 
=& \unm|A|^2 + \unm\la\cas_\ggo(A),A\ra +\sum \la\pm\ad{[X_i,AX_i]},A\ra \\ 
=& \unm|A|^2 + \unm\la\cas_\ggo(A),A\ra.
\end{align*}
The last equality follows from the fact that $A$ is also orthogonal to $\ad{\hg}$.  
\end{proof}

The following nice features of the prescribed Ricci curvature problem on simple Lie groups therefore hold.  

\begin{corollary}\label{simple-thm}
If $G$ is simple, then $g_{\kil}\in\mca^G_{\ricci}$ and $\mca^G_{inv}$ is open and dense in $\mca^G$.  In particular, $g_{\kil}$ is Ricci locally invertible  if $\scalar(g_{\kil})\ne 0$.  
\end{corollary}

\begin{proof}
The first statement follows from Lemma \ref{dRc-simple} and the fact that $\Ker\cas_\ggo|_{\sym(\ggo)}=\RR I$ if $G$ is simple.  The rest of the corollary follows from  Theorem \ref{RLI-thm} and Corollary \ref{RLI-cor2}.  
\end{proof}

We note that $\scalar(g_{\kil})=\unc(\dim{\hg}-3\dim{\qg})$.  For example, it is nonzero for any complex simple Lie group ($\dim{\hg}=\dim{\qg}$) and it vanishes for $G=\SO(7,1)$.

\subsection{Aut-isometry class variations}\label{equiv-dRic}
We study in this section Question (Q4) from Section \ref{quest}.  Recall from Section \ref{equiv-sec} the natural action of the group $\Aut(G/K)$ on $\sym_+(\pg)^K$.  

\begin{lemma}\label{dRic-equiv}
The tangent space of the orbit at $I\in\sym_+(\pg)^K$ is given by
$$
T_I(\Aut(G/K)\cdot I) \subset \left\{ S(D):\underline{D}\in\Der(\ggo/\kg)\right\} \subset\sym(\pg)^K,
$$
and
$$
d\overline{\Ricci}|_I S(D)= D^t\Ricci_{\mu} + \Ricci_{\mu}D, \qquad\forall \underline{D}:=\left[\begin{matrix} \ast&\ast\\ 0&D\end{matrix}\right]\in\Der(\ggo/\kg),
$$
such that $S(D)\in T_I(\Aut(G/K)\cdot I)$
\end{lemma}

\begin{proof}
Recall from Section \ref{equiv-sec} that $\Aut(G/K)\cdot I\subset \Aut(\ggo/\kg)\cdot I$.  If $\underline{f}(t)\in\Aut(\ggo/\kg)$ is a differentiable curve such that $t\in(-\epsilon,\epsilon)$, $\underline{f}(0)=I$ and
$$
\left.\ddt\right|_0\underline{f}(t)=\underline{D}\in\Der(\ggo/\kg), 
$$
then $D=\left.\ddt\right|_0f(t)$ and 
$$
\left.\ddt\right|_0\underline{f}(t)\cdot I = \left.\ddt\right|_0\left(\left(f(t)^{-1}\right)^tf(t)^{-1}\right)^{1/2} = -\unm (D^t+D) = -S(D),
$$
and so the statement about the tangent space follows.

For any $\underline{f}\in\Aut(\ggo/\kg)$ which is the derivative of an element in $\Aut(G/K)$, one has that $\overline{\Ricci}(\underline{f}\cdot I) = \left(f^{-1}\right)^t\Ricci_\mu f^{-1}$.  Indeed, by setting $h_f:=\underline{f}\cdot I = \left(\left(f^{-1}\right)^tf^{-1}\right)^{1/2}$, it follows that  
\begin{align*}
\la\overline{\Ricci}(h_f)\cdot,\cdot\ra =& \ricci_{\la h_f\cdot,h_f\cdot\ra} = \ricci_{\la f^{-1}\cdot,f^{-1}\cdot\ra} =  \ricci_\mu(f^{-1}\cdot,f^{-1}\cdot) \\
=& \la\Ricci_\mu f^{-1}\cdot,f^{-1}\cdot\ra =  \la (f^{-1})^t\Ricci_\mu f^{-1}\cdot,\cdot\ra,
\end{align*}
where the third equality holds because $\underline{f}\in\Aut(\ggo/\kg)$.  It follows that if $\underline{f}(t)\in\Aut(\ggo/\kg)$ is as above, then
$$
d\overline{\Ricci}|_I (-S(D))= \left.\ddt\right|_0\overline{\Ricci}(\underline{f}(t)\cdot I) = \left.\ddt\right|_0\left(f(t)^{-1}\right)^t\Ricci_\mu f(t)^{-1} =  -D^t\Ricci_{\mu} - \Ricci_{\mu}D,
$$
which concludes the proof of the lemma.
\end{proof}

\subsection{Ricci local invertibility}\label{exa-sec2}
We now give some applications of the moving bracket approach, including the above two lemmas, to the study of Ricci local invertibility (see Section \ref{RLI-sec}).  Recall from Theorem \ref{RLI-thm} that if $\Ker d\overline{\Ricci}|_I = \RR I$ and $\scalar(g)\ne 0$, then $g\in\mca_{inv}^G$, and conversely, $\Ker d\overline{\Ricci}|_I = \RR I$ for any $g\in\mca_{inv}^G$.  

\begin{example}
If $(G/K,g)$ is a symmetric space then $\mu_\pg=0$ and so $d\overline{\Ricci}|_I=0$, according to the fact   that the function $\ricci$ is constant: $\overline{\Ricci}(h)=-\unm\kil$ for any $h\in\sym_+(\pg)^K$, or equivalently, $\ricci(g)=-\unm\kil|_{\pg\times\pg}$ for any $g\in\mca^G$.  
\end{example}

\begin{example}
It follows from Lemma \ref{dRic-equiv} that for any Einstein metric $g$, say $\ricci(g)=\rho g$ (i.e., $\Ricci_\mu=\rho I$),  one has that $d\overline{\Ricci}|_I S(D)=2\rho S(D)$ for any $\underline{D}\in\Der(\ggo/\kg)$.  Thus the equivariant isometry class of a non-flat Einstein $g\in\mca^G$ contributes with nothing to the kernel of $d\overline{\Ricci}|_I$.
\end{example}

\begin{example}
Consider the solvable Lie group $M=G$ with Lie algebra $\mu(X_1,X_i)=AX_i$, $i=2,\dots,n$, where $A$ is a traceless symmetric $(n-1)\times (n-1)$ matrix, and the left-invariant metric $g$ such that the basis $\{ X_i\}$ is orthonormal.  It easily follows that $\Ricci_\mu X_1=-\tr{A^2}X_1$ and $\Ricci_\mu X_i=0$ for all $i=2,\dots,n$.  We therefore obtain from Lemma \ref{dRic-equiv} that $d\overline{\Ricci}|_I S(D)=0$ for any $D\in\Der(\ggo)$ given by $DX_1=0$ and $[D|_\ngo,A]=0$, where $\ngo$ is the subspace generated by $X_i$, $i=2,\dots,n$ (moreover, $\overline{\Ricci}(e^{tS(D)})=\overline{\Ricci}(I)$ for any $t$).  In particular, $\dim{\Ker d\overline{\Ricci}|_I}\geq n-2$, and thus $g$ is not Ricci locally invertible for any $n\geq 4$.  This metric $g$ is a Ricci soliton; indeed $\Ricci_\mu+(\tr{A^2})I\in\Der(\mu)$ (see \cite{alek}).   
\end{example}

\begin{example}
Let $M=G$ be the $4$-dimensional nilpotent Lie group with Lie algebra $\mu(X_1,X_2)=X_3$, $\mu(X_1,X_3)=X_4$ and let $g$ be the left-invariant metric on $G$ determined by the inner product $\ip$ such that $\{ X_i\}$ is orthonormal.  It is easy to see that $\Ricci_\mu=\Mm_\mu=\Diag(-1,-\unm,0,\unm)$ and so if $D\in\Der(\ggo)$ is given by $DX_1=X_3$ and $DX_i=0$, $i=2,3,4$, then by Lemma \ref{dRic-equiv}, 
$$
d\overline{\Ricci}|_I S(D)= D^t\Ricci_{\mu} + \Ricci_{\mu}D =0, 
$$
from which follows that $g\notin\mca^G_{inv}$.  However, it is easy to check that $\overline{\Ricci}(e^{tS(D)})$ is not constant in $t$.  Note that $g$ is also a Ricci soliton since $\Ricci_\mu+\frac{3}{2}I\in\Der(\mu)$ (see \cite{alek}).  
\end{example}

In the following examples, the relatively low dimension of $\mca^G$ allows us to analyze the rank of $d\ricci|_g$ in a direct way, with the help of a mathematical software.  

\begin{example}\label{s5s1}
We consider $M=\SSS^5\times\SSS^1$ and $G=\SU(3)\times S^1$, so $K=\SU(2)$.  The reductive decomposition $\ggo=\sug(2)\oplus\pg$, where $\pg=\pg_1\oplus\RR$ and $\pg_1$ is the $\kil_{\sug(3)}$-orthogonal complement of $\sug(2)$ in $\sug(3)$, decomposes in $K$-irreducible subrepresentations as $\pg=\CC^2\oplus\RR_1\oplus\RR$, where $\pg_1=\CC^2\oplus\RR_1$.  Therefore, for each $h \in \sym_+(\pg)^K$ we have,   
$$
h = \left[\begin{smallmatrix}
 a & & & &&&\\   &a & & &&&\\   & & &a &&&\\ &&  &&a&&\\&&&&&b&t\\&&&&&t&c
\end{smallmatrix}\right], \quad a,b,c>0, \quad bc>t^2, \qquad   
\overline{\Ricci}(h)= \left[\begin{smallmatrix}
 \alpha & & & &&&\\   &\alpha & & &&&\\   & & &\alpha &&&\\ &&  &&\alpha&&\\&&&&&\beta&\tau\\ &&&&&\tau&\gamma
\end{smallmatrix}\right], \quad \alpha,\beta,\gamma,\tau\in\RR. 
$$
Using \eqref{Ric}, a straightforward computation gives that
$$
\alpha= \tfrac{4a^2-b^2-t^2}{8a^2}, \quad \beta=\tfrac{(b^2+t^2)^2}{4a^4}, \quad \gamma=\tfrac{(b+c)^2t^2}{4a^4}, \quad \tau= \tfrac{(b^2+t^2)(b+c)t}{4a^4}.
$$
Note that $\beta \gamma = \tau^2$.  Thus the rank of $d\overline{\Ricci}_h$ is always $\leq 2$ and so $\mca^G_{inv}=\emptyset$.  Alternatively, it can be checked that $\dim{\Ker d\overline{\Ricci}_h}\geq 2$ for any $h \in \sym_+(\pg)^K$.  
\end{example}

\begin{example}\label{s5s5}
For $M=\SSS^5\times\SSS^5$ and $G=\SU(3)\times\SU(3)$, one has that $K=\SU(2)\times\SU(2)$ and the reductive decomposition $\ggo=\kg\oplus\pg$, orthogonal relative to $\kil_{\ggo}$.  The decomposition in $K$-irreducible subrepresentations is given by $\pg=(\CC^2)_1\oplus(\CC^2)_2\oplus\RR_1\oplus\RR_2$.  Each $h \in \sym_+(\pg)^K$ is therefore identified with a $5$-upla $(a,b,c,d,t)$, where $a,b,c,d>0$ and $cd>t^2$.   It is straightforward to obtain that $\overline{\Ricci}(h)=(\alpha,\beta,\gamma,\delta,\tau)$, where
$$
\begin{array}{c}
\alpha= \frac{4a^2-c^2-t^2}{8a^2}, \quad \beta=\frac{4b^2-d^2-t^2}{8b^2}, \quad \gamma=\frac{b^4(c^2+t^2)^2 + t^2a^4(c+d)^2 -2a^4b^4}{4a^4b^4} \\ \\
\delta=\frac{a^4(d^2+t^2)^2 + t^2b^4(c+d)^2 -2a^4b^4}{4a^4b^4}, \quad \tau=\frac{(c+d)(c^2 b^4 + t^2 b^4+a^4 d^2+t^2 a^4)t}{4a^4b^4}. 
\end{array}
$$
It is also easy to check that $\dim{\Ker d\overline{\Ricci}_h}=1$ for any $h=(1,1,1,1,t)$, $0<t<1$ and that $\scalar(h)=\tfrac{-2t^6+18 t^4 -34 t^2+10}{4(1-t^2)^2}$, which only vanishes at a single $t_0\in (0,1)$.  This implies that $h$ is Ricci locally invertible for any $t\ne t_0$, so $\mca^G_{inv}$ is nonempty and consequently open and dense.  
\end{example}

\begin{example}\label{s7s5}
Similarly to the above example, we consider $M=\SSS^7\times\SSS^5$ and $G=\SU(4)\times\SU(3)$, so $K=\SU(3)\times\SU(2)$ and $\pg=\CC^3\oplus\CC^2\oplus\RR_1\oplus\RR_2$.  It is straightforward to obtain that the function $\overline{\Ricci}(a,b,c,d,t)=(\alpha,\beta,\gamma,\delta,\tau)$ satisfies that 
$$
d\overline{\Ricci}_{(1,1,1,1,t)}=\left[\begin{smallmatrix} 
0&-\frac{1}{6}& \frac{1}{6} + \frac{1}{6} t^2& 0&-\frac{1}{6}t \\ \frac{1}{4} +  \frac{1}{4} t^2& 0 & 0 & - \frac{1}{4}& - \frac{1}{4}t\\ 
-4 t^2 & 1+2 t^2 & -1-2 t^2-t^4 & t^2 & t (3+t^2) \\ 
-1-2 t^2-t^4 & t^2 & -4 t^2 & 1+2 t^2 & t (3+t^2)\\ -2 t (1+t^2) & \frac{t}{2}  (3+t^2) & -2 t (1+t^2) & \frac{t}{2}  (3+t^2) & 3 t^2+1
\end{smallmatrix}\right], 
$$
and that this matrix has a one-dimensional kernel for any $0<t<1$.  Since the scalar curvature is given by $\scalar(t)=\tfrac{-t^6+12 t^4-23 t^2+8}{2(1-t^2)^2}$, we obtain that $\mca^G_{inv}$ is open and dense in $\mca^G$.  
\end{example}

\begin{example}\label{s3s1}
If $M=\SSS^3\times\SSS^1$ and $G=\SU(2)\times S^1$, then $\mca^G$ is precisely the set of all left-invariant metrics on $M=G$.  Consider the ordered basis $\{X_1,\dots, X_4\}$ of $\ggo=\sug(2) \times \RR$, where $\{X_1,\dots, X_3\}$ is the basis of $\sug(2)$ given in Example \ref{dim3}.  Thus the function $\overline{\Ricci}$ depends on $10$ variables determined by 
$$
h = \left[\begin{smallmatrix}
a & x & y & r \\
x & b & z & s\\ y& z& c&t\\ r&s&t&d 
\end{smallmatrix}\right]\in\sym_+(\ggo). 
$$
A straightforward computation gives that the rank of $d\overline{\Ricci}_h$ is $9$ for $a=1$, $b=c=d=2$, $0<t<2$ and $x=y=z=r=s=0$, and the scalar curvature equals $\scalar(t)=\tfrac{-t^6+14 t^4-73 t^2+60}{16(4-t^2)^2}$.  Thus $\mca^G_{inv}$ is nonempty and so open and dense.  Unlike the two above examples, $\dim{\Ker d\overline{\Ricci}_h}=2$ for any metric with $a=b=c=d=1$, $0<t<1$ and $x=y=z=r=s=0$.
\end{example}

\section{Naturally reductive case}\label{natred-sec}

Let $M^n$ and $G$ be as in Section \ref{preli}.  A metric $g\in\mca^G$ is said to be {\it naturally reductive with respect to} $G$ if there exists a {\it reductive complement} $\pg$ (i.e., a reductive decomposition $\ggo=\kg\oplus\pg$) such that
$$
\la [X,Y]_\pg,Y\ra =0, \qquad\forall X,Y\in\pg,
$$
where $\ip:=g_o\in\sym_+(\pg)^K$.  Equivalently, the map $\ad_\pg{X}:\pg\rightarrow\pg$ is skew-symmetric (i.e., $\exp{tX}\cdot o$ is a geodesic) for any $X\in\pg$.  Note that $G$ is necessarily unimodular.  Since the condition may strongly depend on the reductive complement, we shall make clear sometimes that the metric $g$ is {\it naturally reductive with respect to $G$ and $\pg$} if necessary.  A {\it naturally reductive complement} is the reductive complement of some naturally reductive metric with respect to $G$.  

Naturally reductive spaces were studied by Kostant back in 1956.  The following discussion is strongly influenced by the concepts and results given in \cite{Kst2}.   

For any reductive decomposition $\ggo=\kg\oplus\pg$, 
$$
\overline{\ggo}:=\pg+[\pg,\pg]=[\pg,\pg]_\kg\oplus\pg,
$$ 
is always an ideal of $\ggo$, where the subscript $\kg$ denotes projection on $\kg$ relative to $\ggo=\kg\oplus\pg$.  Moreover, if $g$ is naturally reductive with respect to $G$ and $\pg$, then the normal connected Lie subgroup $\overline{G}\subset G$ with Lie algebra $\overline{\ggo}$ is also transitive on $M$.  Indeed, any point in $M$ can be attained with a geodesic departing from $o$, which must be of the form $\exp{tX}\cdot o$ for some $X\in\pg$ and so $\exp{tX}\in\overline{G}$ for all $t$.  

\begin{definition}\cite{Kst2}  
A reductive complement $\pg$ is called {\it pervasive} when $\overline{\ggo}=\ggo$ (i.e., $[\pg,\pg]_\kg=\kg$).  
\end{definition}

A given homogeneous space can admit pervasive and non-pervasive naturally reductive complements at the same time (see Example \ref{perv-exa} below).  

Note that $g\in\mca^{\overline{G}}$ and it is also a naturally reductive metric on the homogeneous space $M=\overline{G}/\overline{K}$, where $\overline{K}:=\overline{G}\cap K$, with respect to $\overline{G}$ and the pervasive reductive decomposition $\overline{\ggo}=\overline{\kg}\oplus\pg$, where $\overline{\kg}:=[\pg,\pg]_\kg$.  

Under the presence of an $\ad{\ggo}$-invariant symmetric bilinear form $Q$ on $\ggo$ such that $Q(\kg,\pg)=0$ and $\ip=Q|_{\pg\times\pg}$, it is clear that $g$ is naturally reductive with respect to $G$ and $\pg$.  It also follows that $\pg$ is pervasive since the $Q$-orthogonal complement of $\overline{\ggo}$ is an ideal contained in $\kg$ which must vanishes by almost-effectivness.  Moreover, $Q$ is necessarily non-degenerate; indeed, for any $Z\in\kg$ there exists $X\in\pg$ such that $[Z,X]\ne 0$ by almost-effectiveness and so $Q(Z,[X,[Z,X]])=Q([Z,X],[Z,X])>0$.  Recall that if $G$ is compact and $Q>0$ then $g$ is called {\it normal}, and if $G$ is compact semi-simple and $Q=-\kil_{\ggo}$, then $g$ is called {\it standard}.

Remarkably, the converse assertion holds.   It was proved by Kostant in the compact case and an alternative proof was given by D'Atri and Ziller, without assuming $G$ compact.  

\begin{theorem}\cite[Theorem 4]{Kst2}\label{Q-thm}, \cite[p.4]{DtrZll}.
If $g$ is naturally reductive with respect to $G$ and a pervasive $\pg$, then there exists a unique non-degenerate $\ad{\ggo}$-invariant symmetric bilinear form $Q$ on $\ggo$ such that $Q(\kg,\pg)=0$ and $\ip=Q|_{\pg\times\pg}$.  
\end{theorem}

Essentially, what it has to be shown is that the following compelled condition,
$$
Q([X,Y]_\kg,[X',Y']_\kg) = -Q(Y,[X,[X',Y']_\kg]), \qquad \forall X,Y,X',Y\in\pg,
$$
works as a definition of $Q|_{\kg\times\kg}$.  We note that $\pg$ is determined by $Q$ (so by $g$) as the $Q$-orthogonal complement of $\kg$.

\subsection{Irreducibility} 

The following property will be crucial in our study of Ricci locally invertibility among the class of naturally reductive metrics below.  

\begin{definition}\label{pirr}
A reductive complement $\pg$ is called {\it irreducible} if there exists no nontrivial subspace invariant under the space of operators $\ad{\kg}|_\pg+\ad_\pg{\pg}$.  Otherwise, $\pg$ is called {\it reducible}.  
\end{definition}

Let $g$ be a naturally reductive metric on $M$ with respect to $G$ and $\pg$ and set, as usual, $\ip:=g_o\in\sym_+(\pg)^K$.  Note that another metric $g_h=\la h\cdot,h\cdot\ra\in\mca^G$, $h\in\sym_+(\pg)^K$, is naturally reductive with respect to $G$ and $\pg$ if and only if $[h,\ad_\pg{\pg}]=0$.  The following conditions are equivalent:

\begin{enumerate}[{\small $\bullet$}]
\item $\pg$ is irreducible.

\item $g$ is, up to scaling, the unique naturally reductive metric on $M$ with respect to $G$ and $\pg$.

\item $(M,g)$ is holonomy irreducible (see \cite[Theorem 5]{Kst2} and \cite[Theorem 6]{DtrZll}).  Indeed, $\ad{\kg}|_\pg+\ad_\pg{\pg}$ generates the holonomy algebra at $o\in M$ since the natural connection (or canonical connection of the first kind) of the homogeneous space $M=G/K$ is characterized by $\nabla_XY:=\unm[X,Y]_\pg$ for all $X,Y\in\pg$ and coincides with the Levi-Civita connection of any naturally reductive metric on $M$ with respect to $G$ and $\pg$ (see \cite[Theorem 1]{Kst2}).   

\item $(\widetilde{M},g)$ is de Rham irreducible, where $\widetilde{M}$ denotes the simply connected cover of $M$ (see \cite[Corollary 7]{Kst2}).
\end{enumerate}

The following construction of non-pervasive naturally reductive complements is due to Kostant.  Let $\overline{G}$ be a connected Lie group acting transitively on $M$, with isotropy $\overline{K}$.  Suppose that $g\in\mca^{\overline{G}}$ is naturally reductive with respect to $\overline{G}$ and a reductive decomposition $\overline{\ggo}=\overline{\kg}\oplus\pg$ with $\pg$ pervasive.  For each Lie algebra $\ngo$ and monomorphism $\tau:\ngo\rightarrow\overline{\ggo}$ such that 
$$
\tau(\ngo)\subset\lgo:=\left\{ X\in \pg\cap[\overline{\ggo},\overline{\ggo}]:[\overline{\kg},X]=0\right\},
$$
we define $\ggo:=\ngo\oplus\overline{\ggo}$ as a direct sum of ideals.  Using the form $Q$ on $\overline{\ggo}$ provided by Theorem \ref{Q-thm}, it is easy to see that $\lgo$ is a (possibly trivial) Lie subalgebra of $\overline{\ggo}$ and satisfies $[\lgo,\pg]\subset\pg$ (note that $Q(\overline{\kg},[\lgo,\pg])=0$).  Consider the transitive Lie group $G:=N\overline{G}$ of $M$, where $N$ is the connected Lie subgroup of $\overline{G}$ with Lie algebra $\tau(\ngo)$.  The corresponding isotropy subgroup is $K=K_0\overline{K}$, where $K_0=N\cap\overline{K}$ and has Lie algebra $\kg_0:=\{ (X,\tau(X)):X\in\ngo\}\subset\ggo$, so $\kg=\kg_0\oplus\overline{\kg}$.  Thus $\ip$ is $\Ad(K)$-invariant and $g\in\mca^G$ is also naturally reductive with respect to $G$ and the reductive decomposition $\ggo=\kg\oplus\pg=\kg_0\oplus\overline{\kg}\oplus\pg$.  Since $[\pg,\pg]_{\overline{\kg}}=\overline{\kg}$, one has that $[\pg,\pg]_{\kg}=\overline{\kg}$ and hence $\pg$ is not pervasive as a reductive complement of $M=G/K$, unless $\ngo=0$.  

It was shown by Kostant that the converse holds.  

\begin{theorem}\label{perv-const}\cite[Corollary 4]{Kst2}
Any naturally reductive space with a non-pervasive $\pg$ can be constructed in the above way.  
\end{theorem}

\begin{corollary}\label{nonperv}
If $g$ is a naturally reductive metric on $M$ with respect to $G$ and $\pg$, then $\pg$ is irreducible as a reductive complement for $M=G/K$ if and only if it is so as a reductive complement for $M=\overline{G}/\overline{K}$, where $\overline{\ggo}=\overline{\kg}\oplus\pg$ and $\overline{\kg}=[\pg,\pg]_\kg$.
\end{corollary}

\begin{proof}
It follows from the above construction that $\ad{\kg_0}|_\pg\subset \ad_\pg{\pg}$, therefore, 
$$
\ad{\kg}|_\pg+\ad_\pg{\pg} = \ad{\kg_0}|_\pg+\ad{\overline{\kg}}|_\pg+\ad_\pg{\pg} = \ad{\overline{\kg}}|_\pg+\ad_\pg{\pg},
$$
concluding the proof. 
\end{proof}

\begin{corollary}\label{perv-exist}
Assume that $M=G/K$ admits a naturally reductive metric with respect to $G$.  Then there exists a naturally reductive metric on $M$ with respect to $G$ and a pervasive reductive complement $\pg$. 
\end{corollary}

\begin{proof}
In terms of the above construction, we consider $\pg':=\ngo\oplus\pg_1$, where $\pg=\tau(\ngo)\oplus\pg_1$ is the orthogonal decomposition with respect to $g\in\mca^G$, $g_o=\ip\in\sym_+(\pg)^K$, a naturally reductive metric with respect to $G$ and $\pg$.  It  follows that $\ggo=\kg\oplus\pg'$ and $[\kg,\pg']\subset\pg'$, that is, $\pg'$ is a reductive complement of $G/K$.  Moreover, $\pg'$ is pervasive.  Indeed, using that $\tau(\ngo)\subset  [\tau(\ngo),\tau(\ngo)]+[\pg_1,\pg_1]\cap\pg$ (recall that $\tau(\ngo)\subset [\overline{\ggo},\overline{\ggo}]\cap\pg$), one obtains that the projection of $[\pg',\pg']$ on $\kg$ relative to $\ggo=\kg\oplus\pg'$ is the whole $\kg$.  On the other hand, it is easy to check that the metric $g'\in\mca^G$, $g'_o=\ip'\in\sym_+(\pg')^K$, where 
$$
\ip':=\la\tau\cdot,\tau\cdot\ra|_{\tau(\ngo)}+\ip|_{\pg_1\times\pg_1},
$$ 
is naturally reductive with respect to $G$ and $\pg'$, concluding the proof.   
\end{proof}

\begin{example}\label{perv-exa}
Consider $G=\SO(3)\times\SO(5)$ and on its Lie algebra $\sog(3)\oplus\sog(5)$, the basis $\{ X_{ij}\}$ of $\sog(5)$ as in Example \ref{so5-1} and a basis $\{ Y_1,Y_2,Y_3\}$ of $\sog(3)$ with identical Lie brackets as $\{ X_{34}, X_{35}, X_{45}\}$.   We also consider the homogeneous space $M=G/K$ for the connected Lie subgroup $K$ of $G$ with Lie algebra $\kg=\kg_0\oplus\RR X_{12}$, where $\kg_0:=\spann\{ Y_1+X_{34}, Y_2+X_{35}, Y_3+X_{45}\}$.   The reductive complement $\pg:= \spann\{ X_{ij}: ij\ne 12\}$ is not pervasive;  indeed, $\overline{\ggo}=\overline{\kg}\oplus\pg$, where $\overline{\kg}:=\RR X_{12}$.  The corresponding decompositions of $\pg$ as $\Ad(K)$ and $\Ad(\overline{K})$ irreducible representations are respectively given by 
$$
\pg=\RR^6\oplus\RR^3, \qquad \pg=(\RR)^3\oplus(\RR^2)^3,
$$
from which follows that $\dim{\mca^G}=2$ and $\dim{\mca^{\overline{G}}}=12$.  It is easy to check that $\pg$ is irreducible in both cases, in accordance with Corollary \ref{nonperv}.  Note that $-\kil_{\sog(5)}|_{\pg\times\pg}$ defines a (standard) naturally reductive metric on both $M=G/K$ and $M=\overline{G}/\overline{K}$.  

On the other hand, it is easy to check that
$$
\pg':=\spann\{ Y_1,Y_2,Y_3, X_{13}, X_{14}, X_{15}, X_{23}, X_{24}, X_{25}\},
$$
is a pervasive naturally reductive complement of $M=G/K$ (cf. Corollary \ref{perv-exist}).  
\end{example}

We now show that the irreducibility of a pervasive naturally reductive complement only depends on $M=G/K$.  

\begin{definition}\label{dec-def}
A Lie algebra is called {\it decomposable} if it is the direct sum of two of its ideals; otherwise, it is called {\it indecomposable}.  A Lie subalgebra $\kg$ of a Lie algebra $\ggo$ is said to be $\ggo$-{\it decomposable} when $\kg=\kg\cap\ggo_1\oplus\kg\cap\ggo_2$ for some nonzero ideals $\ggo_1$ and $\ggo_2$ of $\ggo$ such that $\ggo=\ggo_1\oplus\ggo_2$.  Otherwise, it is called $\ggo$-{\it indecomposable}.       
\end{definition}

We note that any subalgebra of an indecomposable $\ggo$ is automatically $\ggo$-indecomposable, and that an indecomposable subalgebra of some $\ggo$ may be $\ggo$-decomposable since $\kg\cap\ggo_i$ is allowed to vanish.    

\begin{proposition}\label{pirred-prop2}
Let $g$ be a naturally reductive metric on $M$ with respect to $G$ and the reductive decomposition $\ggo=\kg\oplus\pg$ and assume that $\pg$ is pervasive.  Then $\pg$ is irreducible if and only if $\kg$ is $\ggo$-{\it indecomposable}.   
\end{proposition}

\begin{proof}
Let $Q$ denote the symmetric form attached to $g\in\mca^{G}$ by Theorem \ref{Q-thm}.  If $\pg$ is reducible, then there exists a $\ip$-orthogonal $\ad{\kg}|_\pg$-invariant decomposition $\pg=\pg_1\oplus\pg_2$, $\pg_i\ne 0$, such that $[\pg_i,\pg_i]_\pg\subset\pg_i$ and $[\pg_1,\pg_2]_\pg=0$.  It is easy to see that $[\pg_1,\pg_1]_{\kg}$ and $[\pg_2,\pg_2]_{\kg}$ are $Q$-orthogonal and hence $\ggo=\ggo_1\oplus\ggo_2$, where $\ggo_i:=[\pg_i,\pg_i]_{\kg}\oplus\pg_i$ are both ideals of $\ggo$.  Thus $\kg=[\pg_1,\pg_1]_{\kg}\oplus[\pg_2,\pg_2]_{\kg} = \kg\cap\ggo_1\oplus\kg\cap\ggo_2$, that is, $\kg$ is $\ggo$-decomposable.  

Conversely, if $\kg = \kg\cap\ggo_1\oplus\kg\cap\ggo_2$, where $\ggo=\ggo_1\oplus\ggo_2$ and $\ggo_i$ ideal of $\ggo$, then it is easy to see that $\pg$ is reducible by using that $\pg$ is the $Q$-orthogonal complement of $\kg$ in $\ggo$.  
\end{proof}

The following is a consequence of Corollary \ref{perv-exist} and the above proposition.  

\begin{corollary}\label{comp-irred}
If $M=G/K$ admits a naturally reductive metric with respect to $G$ (e.g., if $G$ is compact) and $\kg$ is $\ggo$-indecomposable, then $M$ admits a naturally reductive metric with respect to $G$ and a pervasive and irreducible $\pg$.  
\end{corollary}

Note that any $M=G/K$ with $G$ compact does admit a naturally reductive metric with respect to $G$, namely any normal metric.

\subsection{Examples}\label{exa-nr}
There are three main classes of naturally reductive spaces, which we next describe.  We refer to \cite{Str} and the references therein for classification results in low dimensions.  

\begin{example}\label{DZ1}
Consider the D'Atri-Ziller metrics on a compact Lie group $M=H$ as described in Example \ref{DZ-PRP}.  It was proved in \cite{DtrZll} that every $g\in\mca^G$ is naturally reductive with respect to $G=H\times K$ and some pervasive reductive decomposition $\ggo=\Delta\kg\oplus\pg_g$, and that in the case when $H$ is simple, these metrics actually exhaust the set of all naturally reductive metrics with respect to some transitive Lie group $G$ on $H$.  It is known that  $\pg_g$ is irreducible for any (or some) $g\in\mca^G$ if and only if $\kg$ is $\hg$-indecomposable (see \cite[Theorem 6]{DtrZll}).  

If $\hg=\lgo\oplus\widetilde{\kg}\oplus\ag$, $\kg=\lgo\oplus\widetilde{\kg}$, where $\lgo$ is an ideal of $\kg$, and $\widetilde{\kg}=\kg_1\oplus \dots \oplus\kg_r$ is any $\ad{\kg}$-invariant decomposition, then for any $\beta,\alpha_i>0$, the $H$-invariant metric
$$
\ip=\alpha_1\ip|_{\kg_1\times\kg_1} + \dots + \alpha_r\ip|_{\kg_r\times\kg_r}+\beta\ip|_{\ag\times\ag}\in\sym_+(\pg)^L, \qquad \pg:=\widetilde{\kg}\oplus\ag,
$$
on the homogeneous space $M=H/L$ is also naturally reductive with respect to $G=H\times K$ and some reductive complement $\pg_g$ (see \cite[Chapter 7]{DtrZll} and \cite{Grd}).
\end{example}

There is a non-compact analogous of the above example. 

\begin{example}\label{DZ2}
On a non-compact semisimple Lie group $M=H$, consider the transitive Lie group $G=H\times K$, where $K\subset H$ is the maximal compact subgroup of $H$.   Thus $M=G/\Delta K$ and $\mca^G$ is identified with the space of all left-invariant metrics on $H$ which are also $K$-invariant.  It also holds in this case that any $g\in\mca^G$ is naturally reductive with respect to $G$ and some pervasive reductive decomposition $\ggo=\Delta\kg\oplus\pg_g$ (see \cite{DtrZll, Grd}).  It is easy to see that $\pg_g$ is irreducible if and only if $H$ is simple.  The version on the homogeneous space $M=H/L$ works as in the compact case above.  
\end{example}

In a quite different context, naturally reductive metrics also show up on nilpotent Lie groups.  

\begin{example}\label{nr-2step}
Let $\tau:K\longrightarrow\End(V)$ be a finite-dimensional representation of a compact Lie group $K$ such that the corresponding representation $\tau:\kg\longrightarrow\End(V)$ is faithful and does not have trivial subrepresentations.  By fixing any $\ad{\kg}$-invariant inner product $\ip_\kg$ on $\kg$ and any $\tau(\kg)$-invariant inner product $\ip_V$ on $V$, one can we define the following $2$-step nilpotent Lie bracket $\lb_\ngo$ on $\ngo:=\kg\oplus V$: $[\kg,\ngo]_\ngo=0$ and
$$
\la [X,Y]_\ngo,Z\ra_\kg := \la\tau(Z)X,Y\ra_V, \qquad \forall X,Y\in V, \quad Z\in\kg.
$$
These algebras were first introduced in \cite{EbrHbr}.  It is easy to see that the isomorphism class of $\ngo$ is independent from the invariant inner-products chosen (see \cite{manus}).  It was shown in \cite{Grd} that the left-invariant metric $\ip_\ngo=\ip_\kg+\ip_V$ on the corresponding simply connected Lie group $N$ is naturally reductive with respect to the Lie group $G:=K\ltimes N$ with Lie algebra $\ggo=\kg\ltimes\ngo$ and the reductive decomposition $\ggo=\kg\oplus\pg$, where $\pg:=\{(Z,Z+X)\in\kg\oplus\ngo:Z\in\kg,\, X\in V\}$.  It is easy to see that if  
$\ngo$ is indecomposable, then $\pg$ is irreducible.  
\end{example}

In \cite[Theorem (6.1)]{Grd}, Gordon obtained structure results which completely describe naturally reductive spaces as kind of "amalgamated products" of the above three types.

\subsection{Ricci curvature}\label{ric-nr}
Let $g\in\mca^G$ be a naturally reductive on $M$ with respect to $G$ and $\ggo=\kg\oplus\pg$, and set $\ip:=g_o$.  Recall from the moving-bracket approach described in Section \ref{mba-sec} the formula for the Ricci operator of $g$ given by $\Ricci_\mu=\Mm_{\mu_p}-\unm\kil_\mu$, where $\mu$ is the Lie bracket of $\ggo$ (see \eqref{Ric}).  It follows from \eqref{mm2} that in the naturally reductive case the moment map takes the following simpler form,
\begin{equation}\label{mm-nr}
\Mm_{\mu_\pg}=\unc\sum (\ad_\pg{X_i})^2, \qquad\mbox{so}\qquad \la\Mm_{\mu_\pg}X,X\ra=-\unc|\ad_\pg{X}|^2,\quad\forall X\in\pg.
\end{equation}

\begin{remark}
In particular, for $G$ compact semisimple, $K$ trivial and $\ip=-\kil$, $\Mm_{\mu}$ is, up to scaling, the Casimir operator acting on the adjoint representation of $\ggo$.
\end{remark}

We consider the operator $\cas_\pg:\sym(\pg)\longrightarrow\sym(\pg)$ defined by
\begin{equation}\label{Cp-def}
\cas_\pg(A):=-\sum[\ad_\pg{X_i},[\ad_\pg{X_i},A]],
\end{equation}
where $\{ X_i\}$ is any orthonormal basis of $(\pg,\ip)$.  Note that $\cas_\pg\geq 0$ and $\cas_\pg(A)=0$ if and only if $[A,\ad_\pg{\pg}]=0$.

\begin{remark}
For $G$ compact semisimple, $K$ trivial and $\ip=-\kil$, $\cas_\ggo$ is precisely the Casimir operator acting on the representation $\sym(\ggo)$ of $\ggo$ given by $\tau(X)A:=[\ad{X},A]$, i.e., $\cas_\ggo=-\sum\tau(X_i)^2$ (cf. Section \ref{simple-sec}).
\end{remark}

The following equivalence follows from the fact that
$$
\Ker\cas_\pg|_{\sym(\pg)^K}=\{ A\in\sym(\pg):[A,\ad{\kg}|_\pg]=0 \;\mbox{and}\;  [A,\ad_\pg{\pg}]=0\};
$$
recall Definition \ref{pirr}.

\begin{lemma}\label{irred2}
$\pg$ is irreducible if and only if $\Ker\cas_\pg|_{\sym(\pg)^K}=\RR I$.
\end{lemma}

The first variation of the moment map (see Lemma \ref{dM}) also simplifies in the naturally reductive case.

\begin{lemma}\label{delta-cas}
If $g$ is naturally reductive with respect to $G$ and $\ggo=\kg\oplus\pg$, then for any $A\in\sym(\pg)$,
$$
d\Mm|_IA=\unm\cas_\pg(A) - A\Mm_{\mu_\pg} - \Mm_{\mu_\pg}A, \qquad\forall A\in\sym(\pg)^K.
$$
\end{lemma}

\begin{proof}
It follows from \eqref{mm-nr} and \eqref{Cp-def} that
\begin{align*}
\la\delta_{\mu_{\pg}}^t\delta_{\mu_{\pg}}(A),A\ra =& |\delta_{\mu_{\pg}}(A)|^2 =
\sum |\ad_{\delta_{\mu_{\pg}}(A)}{X_i}|^2 =  \sum |\ad_\pg{AX_i} + [\ad_\pg{X_i},A]|^2 \\
=&  \sum |\ad_\pg{AX_i}|^2 + |[\ad_\pg{X_i},A]|^2 +2\tr{\ad_\pg{AX_i}[\ad_\pg{X_i},A]} \\
=& -4\sum\la\Mm_{\mu_\pg}AX_i,AX_i\ra + \la\cas_\pg(A),A\ra.
\end{align*}
Thus by Lemma \ref{dM}, 
$$
\la d\Mm|_IA,A\ra = \unm\la\delta_{\mu_{\pg}}^t\delta_{\mu_{\pg}}(A),A\ra = \unm\la\cas_\pg(A),A\ra - 2\tr{\Mm_{\mu_\pg}A^2}, 
$$
concluding the proof.
\end{proof}

\subsection{Ricci first variation}
In this section, we show that the naturally reductive condition considerably simplifies the formula for the first variation of the Ricci curvature.  Many applications will be given in the next subsection.  

Recall from \eqref{Cp-def} the definition of the operator $\cas_\pg:\sym(\pg)\rightarrow\sym(\pg)$.

\begin{lemma}\label{dRic}
If $g$ is naturally reductive with respect to $M=G/K$ and $\ggo=\kg\oplus\pg$, then
$$
d\overline{\Ricci}|_I = \unm \cas_\pg.
$$
\end{lemma}

\begin{remark}
It follows from Corollary \ref{LL} that in the naturally reductive case, the Lichnerowicz Laplacian is simply given by  
$$
\Delta_L\la A\cdot,\cdot\ra = \unm\la \cas_\pg(A)\cdot,\cdot\ra, \qquad \forall  \la A\cdot,\cdot\ra\in \sym^2(\pg)^K \equiv \sca^2(M)^G. 
$$
\end{remark}

\begin{proof}
According to Lemmas \ref{dRicmu} and \ref{delta-cas} and Remark \ref{dM-rem}, for any $A\in\sym(\pg)^K$, 
$$
d\overline{\Ricci}|_I A = d\Mm|_IA + A\Mm_{\mu_\pg} + \Mm_{\mu_\pg}A 
= \unm\cas_\pg(A),
$$
concluding the proof .  
\end{proof}

\begin{proof}[Alternative proof of Lemma \ref{dRic}]
We now give a direct proof of the lemma without using the moving bracket approach.  Given $A\in\sym(\pg)^K$, we set $h(t):=I+tA$ and make the following computations (recall that $g_{h(t)}=\la h(t)\cdot,h(t)\cdot\ra$): 
\begin{align*}
(d\ricci|_I A)(X,X) =& \left.\ddt\right|_0 \ricci(g_{h(t)})(X,X) \\
=& \left.\ddt\right|_0 \Big( -\unm\sum\la h(t)[X,h(t)^{-1}X_i]_\pg, h(t)[X,h(t)^{-1}X_i]_\pg\ra \\
& +\unc\sum\la h(t)[h(t)^{-1}X_i,h(t)^{-1}X_j]_\pg, h(t)X\ra^2 - \unm\kil(X,X) \Big)\\
=& -\unm\left.\ddt\right|_0\sum\la h(t)[X,h(t)^{-1}X_i]_\pg, h(t)[X,h(t)^{-1}X_i]_\pg\ra \\
& +\unc\left.\ddt\right|_0\sum\la h(t)^{-1}[h(t)^{-1}X_i,h(t)^2X]_\pg, h(t)^{-1}[h(t)^{-1}X_i,h(t)^2X]_\pg\ra \\
=& -\sum\la A[X,X_i]_\pg-[X,AX_i]_\pg, [X,X_i]_\pg\ra \\
& +\unm\sum\la -A[X_i,X]_\pg - [AX_i,X]_\pg + 2[X_i,AX]_\pg, [X_i,X]_\pg\ra \\
=& -\sum\la A[X_i,X]_\pg-[AX_i,X]_\pg, [X_i,X]_\pg\ra \\
& +\unm\sum\la -A[X_i,X]_\pg - [AX_i,X]_\pg + 2[X_i,AX]_\pg, [X_i,X]_\pg\ra \\
=& \sum\la -\tfrac{3}{2}A[X_i,X]_\pg + \unm[AX_i,X]_\pg + [X_i,AX]_\pg, [X_i,X]_\pg\ra. 
\end{align*}
On the other hand,
\begin{align*}
\la\cas_\pg(A)X,X\ra =& -\sum \la [\ad_\pg{X_i},[\ad_\pg{X_i},A]]X,X\ra \\
=& -\sum \la [X_i,[\ad_\pg{X_i},A]]X] - [\ad_\pg{X_i},A][X_i,X]_\pg,X\ra \\
=& 2\sum \la [X_i,AX]_\pg - A[X_i,X]_\pg , [X_i,X]_\pg\ra, 
\end{align*}
therefore,
$$
(d\ricci|_I A)(X,X) = \unm\la\cas_\pg(A)X,X\ra + \sum\la -\unm A[X_i,X]_\pg + \unm[AX_i,X]_\pg, [X_i,X]_\pg\ra,
$$
and the right hand summand equals
$$
-\unm\tr{\ad_\pg{X}A\ad_{\pg}{X}} + \unm\tr{(\ad_\pg{X})^2A} = 0,
$$
which concludes the proof.   
\end{proof}

\subsection{Ricci local invertibility}
The following application of Lemma \ref{dRic} follows from the fact that $\Ker d\ricci|_g \simeq \Ker \cas_\pg|_{\sym(\pg)^K}$ and Lemma \ref{irred2}.  

\begin{corollary}\label{dRic-cor}
A naturally reductive metric $g\in\mca^G$ with respect to $G$ and $\pg$ satisfies that $\Ker d\ricci|_g=\RR g$ if and only if $\pg$ is irreducible.  
\end{corollary}

We summarize in the following theorem the main results obtained in this section on the Ricci local invertibility of naturally reductive metrics.  Recall from \eqref{Minv} and \eqref{Minvt} the open subsets of metrics $\mca^G_{\widetilde{\ricci}}\subset\mca^G_{inv}\subset\mca^G$.  

\begin{theorem}\label{dRic-thm2}
Let $g\in\mca^G$ be a naturally reductive metric with respect to $G$ and $\pg$.    

\begin{enumerate}[{\rm (i)}]
\item If $\scalar(g)\ne 0$ and $\pg$ is irreducible, then $g\in\mca^G_{inv}$. 

\item $g\in\mca^G_{\widetilde{\ricci}}$ if and only if $\scalar(g)\ne 0$ and $\pg$ is irreducible. 

\item If $g\in\mca^G_{inv}$, then $\pg$ is irreducible and $(\widetilde{M},g)$ is de Rham irreducible. 

\item If $\pg$ is irreducible, then the sets $\mca^G_{\widetilde{\ricci}}$ and $\mca^G_{inv}$ are open and dense in $\mca^G$.  
\end{enumerate}
\end{theorem}

\begin{proof}
We will use Corollary \ref{dRic-cor} in each of the following arguments.  Part (ii) follows from Theorem \ref{RLI-thm}, (ii) and part (ii) and Theorem \ref{RLI-thm}, (i) implies part (i).  On the other hand, part (iii) follows from \eqref{dRcort} and part (iv) follows from Theorem \ref{RLI-thm}, (iv), concluding the proof.    
\end{proof}

\begin{theorem}\label{cor-dense}
Let $M=G/K$ be a homogeneous space and assume that $M$ admits a naturally reductive metric with respect to $G$ (e.g. if $G$ is compact).  If $\kg$ is $\ggo$-indecomposable, then $\mca^G_{inv}$ is open and dense in $\mca^G$.  
\end{theorem}

\begin{remark}
We have shown in Examples \ref{s5s5}, \ref{s7s5} and \ref{s3s1} that the converse assertion does not hold in general.  
\end{remark}

\begin{proof}
It follows from Proposition \ref{perv-exist} that there exists a naturally reductive $g\in\mca^G$ with respect to $G$ and a pervasive $\pg$, which must be irreducible by Proposition \ref{pirred-prop2}.  We therefore conclude from Theorem \ref{dRic-thm2}, (v) that $\mca^G_{inv}$ is open and dense in $\mca^G$. 
\end{proof}

We finally apply these theorems to D'Atri-Ziller metrics.  

\begin{enumerate}[{\small $\bullet$}]
\item According to Example \ref{DZ1}, if $M=H$ is a compact Lie group and $G=H\times K$, any $g\in\mca^G$ is naturally reductive with respect to $G$ and if in addition $\kg$ is $\hg$-indecomposable, then they are all holonomy irreducible.  It follows from the above theorems that in that case, $\mca^G_{\scalar}\subset\mca^G_{inv}$ and $\mca^G=\mca^G_{\ricci}$.  Thus $\RR_+\ricci(\mca^G)\cap U$ is open for any open subset $U\subset\mca^G_{\scalar}$.  In particular, $\RR_+\ricci(\mca^G)\cap\mca^G$ is open. 

\item In the case when $M=H$ is a non-compact simple Lie group and $G=H\times K$ as in Example \ref{DZ2}, $\mca^G$ also consists of irreducible naturally reductive metrics with respect to $G$.   Therefore,  we also obtain that $\mca^G_{\scalar}\subset\mca^G_{inv}$, $\mca^G=\mca^G_{\ricci}$ and $\RR_+\ricci(\mca^G)\cap U$ is open for any open subset $U\subset\mca^G_{\scalar}$. 
\end{enumerate}

\end{document}